\documentclass[reqno,11pt,centertags]{amsart}
\usepackage[letterpaper,margin=1.5in,bottom=1.3in]{geometry}
\usepackage{amsmath,amsthm,amscd,amssymb,latexsym,enumerate}
\usepackage{overpic}
\usepackage{booktabs}
\usepackage{tikz}
\usetikzlibrary{arrows}
\usetikzlibrary{shapes.multipart}
\usepackage{caption}
\usepackage{subcaption}

\usepackage[final]{hyperref}

\allowdisplaybreaks

\newcommand*{\mailto}[1]{\href{mailto:#1}{\nolinkurl{#1}}}

\makeatletter
\def\theequation{\@arabic\c@equation}

\newcommand{\bbN}{{\mathbb{N}}}
\newcommand{\bbR}{{\mathbb{R}}}

\newcommand{\bbQ}{{\mathbb{Q}}}
\newcommand{\bbZ}{{\mathbb{Z}}}
\newcommand{\bbC}{{\mathbb{C}}}

\newcommand{\cC}{{\mathcal C}}

\newcommand{\no}{\nonumber}
\newcommand{\lb}{\label}
\newcommand{\bi}{\bibitem}
\newcommand{\f}{\frac}

\newcommand{\ol}{\overline}

\newcommand{\hatt}{\widehat}
\newcommand{\dott}{\,\cdot\,}

\newcommand{\Oh}{O}

\renewcommand{\dot}{\overset{\textbf{\Large.}}}
\renewcommand{\ddot}{\overset{\textbf{\Large..}}}

\newcommand{\dom}{\operatorname{dom}}

\newcommand{\supp}{\operatorname{supp}}

\newcommand{\sgn}{\operatorname{sgn}}
\renewcommand{\Re}{\operatorname{Re}}
\renewcommand{\Im}{\operatorname{Im}}
\renewcommand{\ln}{\operatorname{ln}}

\numberwithin{equation}{section}

\newtheorem{theorem}{Theorem}[section]
\newtheorem{lemma}[theorem]{Lemma}

\newtheorem{corollary}[theorem]{Corollary}

\newtheorem{example}[theorem]{Example}
\newtheorem{conjecture}[theorem]{Conjecture}
\theoremstyle{definition}

\newtheorem{remark}[theorem]{Remark}

\usepackage{graphicx}
\DeclareGraphicsExtensions{.pdf}

\begin{document}
\title[Strongly Singular, Ordinary Differential Operators]{Essential Self-Adjointness of Even-Order, Strongly Singular,  Homogeneous Half-Line Differential Operators}

\author{Fritz Gesztesy}
\address{Department of Mathematics,
Baylor University, Sid Richardson Bldg., 1410 S.\,4th Street,
Waco, TX 76706, USA}
\email{\mailto{Fritz\_Gesztesy@baylor.edu}}
\urladdr{\url{https://math.artsandsciences.baylor.edu/person/fritz-gesztesy-phd}}

\author{Markus Hunziker}
\address{Department of Mathematics,
Baylor University, Sid Richardson Bldg., 1410 S.\,4th Street,
Waco, TX 76706, USA}
\email{\mailto{Markus\_Hunziker@baylor.edu}}
\urladdr{\url{https://math.artsandsciences.baylor.edu/person/markus-hunziker-phd}}


\author[G.\ Teschl]{Gerald Teschl}
\address{Faculty of Mathematics\\ University of Vienna\\
Oskar-Morgenstern-Platz 1\\ 1090 Wien\\ Austria}
\email{\mailto{Gerald.Teschl@univie.ac.at}}
\urladdr{\url{http://www.mat.univie.ac.at/~gerald/}} 

\date{\today}

\begin{abstract}
We consider essential self-adjointness on the space $C_0^{\infty}((0,\infty))$ of even order, strongly singular, homogeneous differential operators associated with differential expressions of the type 
\[
\tau_{2n}(c) = (-1)^n \frac{d^{2n}}{d x^{2n}} + \frac{c}{x^{2n}}, \quad x > 0, \; n \in \mathbb{N}, \; c \in \mathbb{R}, 
\]
in $L^2((0,\infty);dx)$. While the special case $n=1$ is classical and it is well-known that 
$\tau_2(c)\big|_{C_0^\infty((0,\infty))}$ is essentially self-adjoint if and only if $c \geq 3/4$, the case $n \in \mathbb{N}$, $n \geq 2$, is far from obvious. In particular, it is not at all clear from the outset that 
\begin{align*}
\begin{split} 
& \text{\it there exists } c_n \in \mathbb{R}, \, n \in \mathbb{N}, \text{\it such that}\\
& \quad \tau_{2n}(c)\big|_{C_0^\infty((0,\infty))} \, \text{\it is essentially self-adjoint if and only if } c \geq c_n.
\end{split}\tag{*}\label{0.1}
\end{align*}

As one of the principal results of this paper we indeed establish the existence of $c_n$, satisfying $c_n \geq (4n-1)!!\big/2^{2n}$, such that property \eqref{0.1} holds.

In sharp contrast to the analogous lower semiboundedness question, 
\[
\text{\it for which values of } c \, \text{\it is } \tau_{2n}(c)\big|_{C_0^{\infty}((0,\infty))} \, \text{\it bounded from below?}, 
\]
which permits the sharp (and explicit) answer $c \geq [(2n -1)!!]^{2}\big/2^{2n}$, $n \in \mathbb{N}$, the answer for \eqref{0.1} is surprisingly complex and involves various aspects of the geometry and analytical theory of polynomials. For completeness we record explicitly, 
\[
c_1 = 3/4, \quad c_2= 45, \quad c_3 = 2240 \big(214+7 \sqrt{1009}\,\big)\big/27,    
\]
and remark that $c_n$ is the root of a polynomial of degree $n-1$. We demonstrate that for $n=6,7$, $c_n$ are algebraic numbers not expressible as radicals over $\mathbb{Q}$ (and conjecture this is in fact true for general $n \geq 6$). 
\end{abstract}

\maketitle

{\scriptsize{\tableofcontents}}
\normalsize

\section{Introduction} \lb{s1}
 
Consider the $2n$th-order differential expression
\begin{align}
\tau_{2n}(c) = (-1)^n \f{d^{2n}}{d x^{2n}} + \f{c}{x^{2n}},  \quad 
x \in (0,\infty), \; n \in \bbN, \; c \in \bbR,       \lb{1.1}
\end{align}
and introduce the underlying preminimal and symmetric $L^2((0,\infty); dx)$-realization 
\begin{equation}
\tau_{2n}(c)\big|_{C_0^{\infty}((0,\infty))}     \lb{1.2} 
\end{equation}
and its closure, the associated minimal operator $T_{2n, min}(c)$ in $L^2((0,\infty); dx)$,
\begin{equation}
T_{2n, min}(c) = \ol{\tau_{2n}(c)\big|_{C_0^{\infty}((0,\infty))}}.      \lb{1.3} 
\end{equation}

The principal question to be posed and answered in this paper is the following:
\begin{align} 
& \text{\it For which values of $c \in \bbR$ is $T_{2n, min}(c)$ self-adjoint $($equivalently,}   \no \\
& \quad \text{\it for which values of $c \in \bbR$ is $\tau_{2n}(c)\big|_{C_0^{\infty}((0,\infty))}$ essentially self-adjoint\,$)$}    \lb{1.4} \\
& \quad \text{\it in $L^2((0,\infty); dx)$?}     \no 
\end{align} 

For the notion of (essentially) self-adjoint Hilbert space operators see, for instance, \cite[Sect.~V.3]{Ka80}, \cite[Sect.~VIII.2]{RS80}, \cite[Sect.~3.2]{Sc12}, and \cite[Sects.~4.4, 5.3]{We80}.

In the special case $n=1$ it is well-known that the precise answer is (see, e.g., \cite{Se49}), 
\begin{equation}
\text{\it \eqref{1.4} holds for $n=1$ if and only if } \, c \geq c_1 = 3/4.    \lb{1.5} 
\end{equation}

{\it A priori} it is not clear at all that this extends to $n \in \bbN$, $n \geq 2$, that is, it is not obvious from the outset that 
\begin{align} 
\begin{split} 
& \text{\it there exists $c_n \in \bbR$, $n \in \bbN$, such that}  \lb{1.6} \\
& \quad \tau_{2n}(c)\big|_{C_0^{\infty}((0,\infty))} \, \text{\it is essentially self-adjoint if and only if $c \geq c_n$.}   
\end{split} 
\end{align}  

Our principal new results, Theorem \ref{t4.6} and Corollary \ref{c4.7} assert that \eqref{1.6} indeed holds for some $c _n \in \bbR$ satisfying
\begin{equation} 
c_n \geq (4n-1)!!\big/2^{2n}, \quad n \in \bbN.      \lb{1.7}
\end{equation}

The proof of the existence of $c_n$ in \eqref{1.6} (satisfying \eqref{1.7}) is surprisingly complex and involves various aspects of the geometry and analytical theory of polynomials. Explicitly, one obtains 
\begin{align} 
\begin{split} 
c_{1} &= 3/4, \quad c_{2 }= 45, \quad 
c_{3 } = 2240 \big(214+7 \sqrt{1009}\,\big)\big/27,     \\
c_{4 } &= 2835 \Bigg( 13711+\f{190309441}{\sqrt[3]{2625188010911+1805760  
   \sqrt{-292868607}}}       \lb{1.8} \\
  &\quad +  \sqrt[3]{2625188010911+1805760   \sqrt{-292868607}}\ \Bigg) 
  \end{split} 
\end{align}
and we note that in this context that $c_n$ is the root of a polynomial of degree $n-1$. In addition, we demonstrate that for $n=6,7$, $c_n$ are algebraic numbers not expressible as radicals over $\bbQ$; we conjecture that this actually continues to hold for general $n \geq 6$. 

Before explaining some of the strategy behind the proof of the existence of $c_n$, and for the purpose of comparison and exhibition of a sharp contrast to the essential self-adjointness problem \eqref{1.6}, we briefly record the precise borderline of semiboundedness of the minimal operator $T_{2n, min}(c)$, which permits a remarkably simple and explicit solution as follows:  
\begin{align}
\begin{split} 
& \text{\it $T_{2n, min}(c)$ is bounded from below, and then actually, $T_{2n, min}(c) \geq 0$, $n \in \bbN$,} \\
& \quad \text{\it if and only if } \, c \geq - \f{[(2n -1)!!]^{2}}{2^{2n}}.     \lb{1.9} 
\end{split} 
\end{align}
This is a consequence of the sequence of sharp Birman--Hardy--Rellich inequalities, see Birman \cite[p.~46]{Bi66} (see also Glazman \cite[p.~83--84]{Gl65}) 
\begin{align}
\begin{split} 
\int_0^{\infty} dx \, \big| f^{(n )}(x)\big|^{2} \geq \f{[(2n -1)!!]^{2}}{2^{2n}} 
\int_0^{\infty} dx \, x^{-2n} |f(x)|^2,&   \lb{1.10} \\
f\in C_{0}^{n }((0,\infty)), \; n  \in \bbN.&
\end{split} 
\end{align}
For more details on \eqref{1.10} see \cite{GLMW18} and the extensive literature cited therein. 

Returning to \eqref{1.6}, our subject at hand, we recall that $\tau_{2n}(c)\big|_{C_0^{\infty}((0,\infty))}$ is essentially self-adjoint in $L^2((0,\infty); dx)$ if and only if $\tau_{2n}(c)\big|_{C_0^{\infty}((0,\infty))}$ is in the limit point case at $x=0$ and $x = \infty$. However, since for all $c \in \bbR$, $c x^{-2n}$ is bounded on $(\varepsilon,\infty)$ for all $\varepsilon > 0$, $\tau_{2n}(c)\big|_{C_0^{\infty}((0,\infty))}$ is automatically in the limit point case at $x = \infty$ and hence it suffices to exclusively focus on whether or not $\tau_{2n}(c)\big|_{C_0^{\infty}((0,\infty))}$ is in the limit point case at $x=0$. 

In this context one observes that $\tau_{2n}(c)\big|_{C_0^{\infty}((0,\infty))}$ is said to be in the {\it limit point case} at an interval endpoint $a \in \{0, \infty\}$ if precisely $n$ solutions of 
\begin{equation} 
\tau_{2n}(c) y(\mu,\dott;c) = \mu y(\mu,\dott;c)    \lb{1.11}
\end{equation} 
(i.e., precisely half of the solutions) lie in $L^2(I_a; dx)$, where $I_a$ is an interval of the type $I_0 = (0,d)$ if $a=0$, and $I_{\infty} = (d,\infty)$ if $a=\infty$, for some fixed $d \in (0,\infty)$.

To decide the limit point property of $\tau_{2n}(c)\big|_{C_0^{\infty}((0,\infty))}$ at $x=0$, one next argues that it suffices to choose $\mu = 0$ in \eqref{1.11} which then leads to a special Euler-type equations which generically has solutions of power-type
\begin{equation}
y_j(0,x;c) = C_j x^{\alpha_j(c)}, \quad 1 \leq j \leq 2n,    \lb{1.12}
\end{equation}
with $\alpha_j(c)$, $1 \leq j \leq 2n$, being the solutions of the underlying discriminant or indicial equation, 
\begin{equation}
D_{2n}(z;c) = \prod_{j=1}^{2n} [z - (j-1)] + (-1)^n c = 0, \quad z \in \bbC.    \lb{1.13} 
\end{equation}
In exceptional cases, where some of the $\alpha_k(c)$ coincide, \eqref{1.12} is replaced by 
\begin{equation}
y_k(0,x;c) = C_k x^{\alpha_k(c)} P(\ln(x)),      \lb{1.14}
\end{equation}
where $P(\dott)$ is a polynomial of degree at most $2n-1$. Since we are interested in whether or not 
$y_j(0,x;c) \in L^2((0,d);dx)$ for some $d \in (0,\infty)$, the presence of logarithmic terms is irrelevant and the deciding $L^2$-criterion for solutions of $\tau_{2n}(c) y(\mu,\dott;c) = 0$ simply becomes
\begin{align}
\begin{split} 
& \Re(\alpha_j(c)) > - 1/2, \, \text{ for $L^2$-membership}, \\
& \quad \text{respectively, } \, \Re(\alpha_j(c)) \leq - 1/2, \, \text{ for non-$L^2$-membership.}      \lb{1.15} 
\end{split} 
\end{align}

In conclusion, to settle the essential self-adjointness problem \eqref{1.6} one needs to establish the existence of $c_n \in \bbR$ such that precisely $n$ roots $\alpha_j(c)$ of $D_{2n}(\dott;c)=0$ satisfy 
$\Re(\alpha_j(c)) \leq - 1/2$ for $c \geq c_n$. (Equivalently, precisely $n$ roots $\alpha_k(c)$ of 
$D_{2n}(\dott;c)=0$ satisfy $\Re(\alpha_k(c)) > - 1/2$ for $c \geq c_n$.)

Turning briefly to the content of each section, we note that Section \ref{s2} introduces minimal and maximal operators associated with general differential expressions $\tau_{2n}$ of order $2n$, $n \in \bbN$, in $L^2((0,\infty); dx)$ and reviews the underlying facts on deficiency indices of the minimal operator 
$T_{2n,min}$, including Kodaira's decomposition principle. Section \ref{s3} discusses perturbed Euler differential systems and investigates the underlying deficiency indices for the minimal operator associated with $\tau_{2n}(c)$ in \eqref{1.1}. In addition, some of the basic theory of first-order systems in the complex domain going back to Fuchs, Frobenius, and Sauvage, in versions championed by Hille and Kneser, is summarized. Moreover, the special examples $\tau_2(c)$ and $\tau_4(c)$ are treated explicitly. Properties of the (real part of the) roots $\alpha_j(c)$ of $D_{2n}(\dott;c)=0$ are the center piece of our principal Section \ref{s4}, culminating in Theorem \ref{t4.6} and Corollary \ref{c4.7} which settle the essential self-adjointness problem \eqref{1.6}. The techniques involved are related to the Grace--Haewood theorem \cite[p.~126]{RS02}, the Routh--Hurwitz criterion, and Orlando's formula \cite[\S~XV.7]{Ga59}. Appendix \ref{sA} shows with the help of Galois theory that $c_6, c_7$ are algebraic numbers that cannot be expressed as radicals over $\bbQ$; we conjecture this actually remains the case for all $c_n$, $n \in \bbN$, $n \geq 6$. 

Finally, some remarks on the notation employed: We denote by $\bbC^{M \times N}$, $M, N \in \bbN$, the linear space of $M \times N$ matrices with complex-valued entries. $I_N$ represents the identity matrix in $\bbC^N$. The spectrum of a matrix (or closed operator in a Hilbert space) $T$ is denoted by $\sigma(T)$.  
The abbreviation $\bbN_0= \bbN \cup \{0\}$ is used.

\section{The Deficiency Indices of $T_{2n, min}(c)$} \lb{s2}

In this section we briefly recall the notions of deficiency indices and limit point, respectively, limit circle cases associated with maximally defined differential operators, generated by formally symmetric differential expressions $\tau_{2n}$ on intervals $(a,b) \subseteq \bbR$, of even order $2n$, $n \in \bbN$, and then specialize the results to the particular case 
$\tau_{2n}(c)$ at hand. We will primarily follow \cite[Sects.~XIII.2, XIII.6]{DS88}, \cite[Sects.~17.4, 17.5]{Na68}, \cite[Sects.~3, 4]{We87} and also refer to \cite[\S\,126]{AG81}, \cite{HS83}, \cite{HS84}, \cite{KR74}, \cite{LM03}, \cite[Chs.~2--4]{WZ19} for relevant background material. 

Assuming $(a,b) \subseteq \bbR$ we suppose that 
\begin{align}
& p_m, r \text{ are (Lebesgue) measurable and real-valued~a.e.~on $(a,b)$}, \quad 0 \leq m \leq n,    \no \\
& p_n > 0, \; r > 0 \text{ (Lebesgue) a.e.~on $(a,b)$},     \lb{2.1} \\
& (1/p_n), \; p_m \in L^1_{loc}((a,b); dx), \quad 0 \leq m \leq n-1,    \no 
\end{align}
and introduce the quasi-derivatives
\begin{align}
& u^{[0]} = u, \; u^{[m]} = u^{(m)}, \quad 0 \leq m \leq n-1,     \no \\
& u^{[n]} = p_n \big(u^{{n-1}}\big)',    \no \\
& u^{[n+1]} = - \big(u^{{n}}\big)' + p_{n-1} u^{{n-1}},     \lb{2.2} \\
& u^{[n+j]} = - \big(u^{{n+j-1}}\big)' + p_{n-j} u^{{n-j}}, \quad 2 \leq j \leq n -1,    \no \\
& u^{[2n]} = - \big(u^{{2n - 1}}\big)' + p_0 u = r (\tau_{2n} u).     \no 
\end{align}
Here the formally symmetric differential expression $\tau_{2n}$ of order $2n$ is given by
\begin{equation}
(\tau_{2n} u)(x) = \sum_{m=0}^n (-1)^m \big(p_m(x) y^{(m)}(x)\big)^{(m)}, \quad x \in (a,b).      \lb{2.3}
\end{equation}

Given \eqref{2.1}--\eqref{2.3}, the maximal $L^2((a,b); rdx)$-realization (in short, the maximal operator), $T_{2n,max}$, associated with $\tau_{2n}$ 
is then defined by
\begin{align}
& T_{2n, max} f = \tau_{2n} f,    \no \\
& f \in \dom(T_{2n, max}) = \big\{g \in L^2((a,b);rdx) \, \big| \, g^{[\ell]} \in AC_{loc}((a,b)), \, 0 \leq \ell \leq 2n - 1; \no \\
& \hspace*{8cm} \tau_{2n} g \in L^2((a,b);rdx)\big\}.     \lb{2.4} 
\end{align}
Introducing the preminimal operator 
\begin{align}
\begin{split}
& \dot T_{2n, min} f = \tau_{2n} f,     \\
& f \in \dom\big(T_{2n, min}\big) = \{g \in \dom(T_{2n, max}) \, | \, \supp\,(g) \, \text{compact}\}      \lb{2.5} 
\end{split} 
\end{align}
in $L^2((a,b); rdx)$, one can show that $\dot T_{2n, min}$ is densely defined, symmetric, and closable. Hence, defining the minimal operator  in $L^2((a,b); rdx)$ associated with $\tau_{2n}$ as the closure of $\dot T_{2n, min}$, 
\begin{equation}
T_{2n, min} = \ol{\dot T_{2n, min}}, 
\end{equation}
one can prove the well-known fact 
\begin{equation}
T_{2n, min}^* = T_{2n, max}, \quad T_{2n, max}^* = T_{2n, min},
\end{equation}
and thus $T_{2n, max}$ is closed. Moreover, if 
\begin{equation}
p_m \in C^m((a,b)), \quad 0 \leq m \leq n,     \lb{2.8} 
\end{equation} 
one can introduce
\begin{equation}
\ddot T_{2n, min} = \tau_{2n}\big|_{C_0^{\infty}((a,b))}, 
\end{equation}
and then also obtains
\begin{equation}
\ol{\ddot T_{2n, min}} = \ol{\dot T_{2n, min}} = T_{2n, min}.  
\end{equation}
Introducing the Lagrange bracket
\begin{equation}
[u,v]_x = \sum_{j=1}^n \big[u^{[j-1]}(x) v^{[2n-j]}(x) - u^{[2n-j]}(x) v^{[j-1]}(x)\big], \quad x \in (a,b),
\end{equation}
one infers for $(d,e) \subset (a,b)$ Lagrange's identity via integrations by parts
\begin{equation}
\int_d^e r(x) dx \, \big\{\ol{(\tau_{2n} u)(x)} v(x) - \ol{u(x)} (\tau_{2n} v)(x)\big\} = [\ol{u},v]_e - [\ol{u},v]_d. 
= [u,v]_x\big|_{x=d}^e. 
\end{equation}
Moreover, if $u(\ol \mu, \dott)$ and $v(\mu, \dott)$ are solutions of 
\begin{equation}
(\tau_{2n} u(\ol \mu, \dott))(x) = \ol \mu u(\ol \mu,x), \quad (\tau_{2n} v(\mu, \dott))(x) = \mu v(\mu,x), 
\quad \mu \in \bbC, \; x \in (a,b),
\end{equation}
then
\begin{equation}
\f{d}{dx} [\ol{u(\ol \mu, \dott)}, v(\mu,\dott)]_x = 0, \quad x \in (a,b). 
\end{equation}
Finally, we also recall the known fact,
\begin{align}
\begin{split} 
& \dom(T_{2n, min}) = \{g \in \dom(T_{2n,max}) \, | \,\text{for all $h  \in \dom(T_{2n,max})$:}  \\ 
& \hspace*{6.65cm}       [h,g]_a = 0 = [h,g]_b \, \}. 
\end{split} 
\end{align}

In the following, the number of $L^2((a,b); rdx)$-solutions $u(\mu_{\pm}, \dott)$ of 
\begin{equation} 
\tau_{2n} u(\mu_{\pm}, \dott) = \mu_{\pm} u(\mu_{\pm}, \dott), \, \text{ with $\pm \Im(\mu_{\pm}) > 0$}, 
\end{equation} 
is denoted by $n_{\pm} (T_{2n, min})$ and called the {\it deficiency indices of} $T_{2n, min}$. This notion is well-defined as $n_{\pm} (T_{2n, min})$ is known to be constant throughout the open complex upper and lower half-plane. As a result, one typically chooses $\mu_{\pm} = \pm i$. Since the coefficients of $\tau_{2n}$ are real-valued, one obtains by a result of von Neumann \cite{vN30} that 
\begin{equation}
n_+ (T_{2n, min}) = n_- (T_{2n, min}),      \lb{2.17} 
\end{equation}
and hence,
\begin{equation} 
0 \leq n_{\pm} (T_{2n, min}) \leq 2n. 
\end{equation}

Finally, given $d \in (a,b)$, and denoting by $T_{2n, min (max),(a,d)}$ and $T_{min (2n, max),(d,b)}$ the corresponding minimal or maximal operator with the interval $(a,b)$ replaced by $(a,d)$ and $(d,b)$, respectively, where $d$ is now a regular endpoint for $\tau_{2n}\big|_{(a,d)}$ and $\tau_{2n}\big|_{(d,b)}$, one has (cf.\ \cite[p.~483--484]{AG81})
\begin{align}
& n_+ (T_{2n, min, (a,d)}) = n_- (T_{2n, min, (a,d)}), \quad n_+ (T_{2n, min, (d,b)}) = n_- (T_{2n, min, (d,b)}),   \no \\
& n \leq n_{\pm} (T_{2n, min, (a,d)}) \leq 2n, \quad n \leq n_{\pm} (T_{2n, min, (d,b)}) \leq 2n,      \lb{2.19} 
\end{align}
and the Kodaira decomposition principle (see, e.g., \cite[Corollary~XIII.2.26]{DS88}, \cite[p.~72]{Na68})
\begin{equation}
n_{\pm} (T_{2n, min}) = n_{\pm} (T_{2n, min, (a,d)}) + n_{\pm} (T_{2n, min, (d,b)}) - 2n      \lb{2.20} 
\end{equation} 
holds. 

\begin{remark} \lb{r2.1}
Given the fact that $d \in (a,b)$ is a regular endpoint for $\tau_{2n}|_{(a,d)}$ and $\tau_{2n}|_{(d,b)}$, the particular (and extreme) case where 
\begin{equation}
n_{\pm} (T_{2n, min, (a,d)}) = n \, \text{ (resp., $n_{\pm} (T_{2n, min, (d,b)}) = n$)}
\end{equation}
is the precise analog of Weyl's {\it limit point case} at $x=a$ (resp., $x=b$) in the classical second order case $n=1$, that is, for $\tau_2|_{(a,d)}$ (resp., $\tau_2|_{(d,b)}$). Hence, we will apply this limit point terminology also in the 
$2n$th-order context in the following. In particular, if 
\begin{equation}
n_{\pm} (T_{2n, min, (a,d)}) = n = n_{\pm} (T_{2n, min, (d,b)}),      \lb{2.22} 
\end{equation}
then $\tau_{2n}|_{(a,b)}$ is in the limit point case at $a$ and $b$ and \eqref{2.20} yields accordingly that  
\begin{equation}
n_{\pm} (T_{2n, min}) = 0     \lb{2.23}
\end{equation}
in this case. Thus, \eqref{2.22}, and hence \eqref{2.23}, is equivalent to
\begin{equation}
T_{2n, min} = T_{2n, max} \, \text{ is self-adjoint in $L^2((a,b); rdx)$,}     \lb{2.24} 
\end{equation}
which in turn is equivalent to 
\begin{equation}
\dot T_{2n, min} \, \text{ is essentially self-adjoint in $L^2((a,b); rdx)$.}     \lb{2.25}
\end{equation}
If in addition hypothesis \eqref{2.8} holds, then each of \eqref{2.22}--\eqref{2.25} is also equivalent to 
\begin{equation}
\ddot T_{2n, min} \, \text{ is essentially self-adjoint in $L^2((a,b); rdx)$.}      \lb{2.26} 
\end{equation}

All other cases, where $1 \leq n_{\pm} (T_{2n, min}) \leq 2n$, describe various degrees of limit circle cases of 
$\tau_{2n}$, with $n_{\pm} (T_{2n, min}) = 2n$ representing the extreme case. 
\hfill $\diamond$
\end{remark}

In the bulk of this paper we are particularly interested in the special case where 
\begin{equation} 
p_n(x) =1, \quad p_m(x) = 0, \;1 \leq m \leq n-1, \quad p_0(x) = c x^{- 2n}, \quad r(x) = 1, 
\quad x \in (0,\infty), 
\end{equation} 
that is, in the concrete example 
\begin{equation}
\tau_{2n}(c) = (-1)^n \f{d^{2n}}{d x^{2n}} + \f{c}{x^{2n}}, \quad x \in (0,\infty), \; n \in \bbN, \; c \in \bbR, 
\end{equation}
denoting the associated (pre)minimal and maximal operators in $L^2((0,\infty);dx)$ by $T_{2n, min}(c)$, 
$\dot T_{2n, min}(c)$, $\ddot T_{2n, min}(c)$, $T_{2n,max}(c)$, etc. 

In particular, we are interested in the question, 
\begin{align} 
\begin{split} 
& \text{``for which values of $c \in \bbR$ is $T_{2n, min}(c)$ self-adjoint}   \\
& \quad \text{\big(resp., $\ddot T_{2n, min}(c)$ essentially self-adjoint\big) in $L^2((0,\infty);dx)$?''}
\end{split} 
\end{align}

\section{Perturbed Euler Differential Systems  and Their Deficienciy Indices} \lb{s3}

In this section we will prove that it suffices to focus on the spectral parameter $\mu = 0$ when trying to determine the number of $L^2((0,d);dx)$-solutions $y(\mu,\dott)$ of 
\begin{align}
\begin{split} 
\tau_{2n}(c) y(\mu,x) = (-1)^n y^{(2n)}(\mu,x) + c x^{-2n} y(\mu,x) = \mu y(\mu,x),&    \lb{3.1} \\ 
x \in (0,d), \; \mu \in \bbC, \; n \in \bbN, \; c \in \bbR,&
\end{split} 
\end{align}
for fixed $d \in (0,\infty)$ (e.g., one could simply choose $d=1$). 
In particular, the deficiency indices of the underlying minimal differential operator $T_{2n, min}(c)$ can be determined from the knowledge of the number of $L^2((0,d);dx)$-solutions of $y(0,\dott)$, that is, one can reduce \eqref{3.1} to the far simpler case $\mu = 0$.

To prove the $\mu$-independence of the number of $L^2((0,d);dx)$-solutions $y(\mu,\dott)$ of \eqref{3.1}, we find it convenient to employ a bit of the celebrated theory of regular singular points (singular points of the first kind) for first-order systems of differential equations in the complex domain, going back to G.\ Frobenius \cite{Fr73}, L.\ Fuchs  \cite{Fu65}, \cite{Fu66}, and L.\ Sauvage \cite{Sa86}, \cite{Sa94}. The theory is aptly summarized in a number of treatises, we just mention 
\cite[p.~17--36]{Ba00}, \cite[p.~108--135]{CL85}, \cite[148--164]{Ga59}, \cite[p.~70--92]{Ha82}, \cite[p.~105--131]{He77}, \cite{Hi26}, \cite[p.~182--198]{Hi69}, \cite[p.~342--352]{Hi97}, 
\cite[p.~356--372, Ch.~XVI]{In56}, \cite[Ch.~V]{Po36}, \cite[Ch.~4]{Te12}, and \cite[216--235]{Wa98}.

In the following $\zeta \in \bbC\backslash\{0\}$ (resp., $\zeta \in D(0;R) \backslash\{0\} 
= \{\zeta \in \bbC \,|\, 0 < |\zeta| < R\}$ for some fixed 
$R \in (0,\infty)$) represents the complex analog of $x \in (0,d)$ in \eqref{3.1} and we will study first-order systems of differential equations of the particular form
\begin{equation}
Y'(\zeta) = \zeta^{-1} A(\zeta) Y(\zeta),   \lb{3.2} 
\end{equation} 
where $Y(\dott)$ represents either an $N \times 1$ solution vector or an $N \times N$ solution matrix, $N \in \bbN$, which generally is multi-valued, and $A(\dott)$ is an $N \times N$ entire (resp., analytic in $D(0;R)$) matrix-valued function, 
\begin{equation}
A(\zeta) = \sum_{m \in \bbN_0} A_m \, \zeta^m. \lb{3.3} 
\end{equation}
The very special structure (at most a first-order pole of the coefficient matrix at $z = 0$) of the right-hand side of \eqref{3.2} then leads to a rather special structure of solutions as described in the following. 

As a warm up we briefly discuss the pure Euler situation where $A(\dott)$ is actually a constant matrix 
$A_0 \in \bbC^{N \times N}$, that is, we consider
\begin{equation}
Y'(\zeta) = \zeta^{-1} A_0 Y(\zeta),   \lb{3.4} 
\end{equation}  
with fundamental (typically, many-valued) matrix solutions of the form 
\begin{equation}
Y(\zeta) = \zeta^{A_0} C = e^{A_0 \ln(\zeta)} C,    \lb{3.5}
\end{equation}
where $C \in \bbC^{N \times N}$ is nonsingular (i.e., ${\det}_{\bbC^N}(C) \neq 0$). Transforming $A_0$ into its Jordan normal form $\hatt A_0 = T A_0 T^{-1}$ for some nonsingular $T \in \bbC^{N \times N}$, and setting 
$\hatt Y(\dott) = T Y(\dott)$ yields 
\begin{equation}
\hatt Y'(\zeta) = \zeta^{-1} \hatt A_0 \hatt Y(\zeta),   \lb{3.6} 
\end{equation}  
hence one can assume without loss of generality that $A_0$ is in Jordan normal form. In this case $A_0$ is represented as a block diagonal matrix consisting possibly of a diagonal matrix $D$ and possibly of a number of nontrivial Jordan blocks of varying $r \times r$, $1 \leq r \leq N$, sizes, denoted by $J_r(\alpha_q)$. In particular, if $J_r(\alpha_q)$ is of the form 
\begin{equation}
J_r(\alpha_q) = \begin{pmatrix}
\alpha_q & 1 & 0 & \cdots & 0 \\
0 & \alpha_q & 1 & \cdots & 0 \\
\vdots  & \vdots  & \ddots & \ddots & \vdots  \\
0 & 0 & 0 &\cdots & 1 \\ 
0 & 0& 0 & \cdots & \alpha_q 
\end{pmatrix}, \quad \alpha_q \in \sigma(A_0),       \lb{3.7} 
\end{equation}
then
\begin{equation}
\zeta^{J_r(\alpha_q)} = \zeta^{\alpha_q} \begin{pmatrix}
1& \ln(\zeta) & [\ln(\zeta)]^2/[2!] & \cdots & [\ln(\zeta)]^{r-1}/[(r-1)!] \\
0 & 1 & \ln(\zeta) & \cdots & [\ln(\zeta)]^{r-2}/[(r-2)!]  \\
\vdots  & \vdots  & \ddots & \ddots & \vdots  \\
0 & 0 & 0 &\cdots & \ln(\zeta) \\ 
0 & 0& 0 & \cdots & 1 
\end{pmatrix},       \lb{3.8} 
\end{equation}
explicitly demonstrating the appearance of powers of logarithms of $\zeta$ in \eqref{3.5} in the case where $A_0$ has an eigenvalue $\alpha_q$ whose algebraic multiplicity strictly exceeds its geometric one. In particular, the eigenvalues 
$\alpha_q$ of $A_0$ are determined via the characteristic equation for $A_0$, also called the {\it indicial equation}, 
\begin{equation}
D_N(z) = {\det}_{\bbC^N} (z I_N - A_0) = 0, \quad z \in \bbC.     \lb{3.9} 
\end{equation}

The general, or perturbed, Euler case \eqref{3.2} leads to analogous results as follows.

\begin{theorem}[Hille \cite{Hi69}, p.~192--198, Kneser \cite{Kn50}]\lb{t3.1} ${}$ \\ 
Given the matrix $A(\dott) \in \bbC^{N \times N}$ in \eqref{3.3} entire $($resp., analytic in $D(0;R)$$)$, the perturbed Euler differential system \eqref{3.2} has a fundamental set of $($generally, multi-valued\,$)$ solutions $Y_j \in \bbC^{N \times 1}$, $j = 1,\dots, N$, of the form,
\begin{equation}
Y_j(\zeta; q) = \sum_{m \in \bbN_0} p_{j,m,q}(\ln(\zeta) \, \zeta^{m + \alpha_q}, \quad 1 \leq j \leq N,   
\lb{3.10} 
\end{equation}
where $\alpha_q$ runs through all distinct eigenvalues of $A_0$ $($i.e., all elements of $\sigma(A_0)$$)$, determined via $D_N (\dott) = 0$, and $p_{j,m,q} (\dott) \in \bbC^{N \times 1}$ are polynomials of degree less than or equal to $N-1$.~The series in \eqref{3.10} converges for $0 < |\zeta| < \infty$ $($resp., for $0 < |\zeta| < R$$)$.
 \end{theorem}
 
 In this context we also refer to Sections~4.3, 4.4, particularly, Theorem 4.11, in Teschl \cite{Te12}, for a succinct treatment of the Frobenius method for first-order systems with a pole structure as in \eqref{3.2}. 

We also note that a fundamental matrix solution of \eqref{3.2} can be obtained in analogy to \eqref{3.5} in the pure Euler case. In particular, under the spectral hypothesis that  
\begin{equation}
\sigma(A_0) \cap \{\sigma(A_0) + \bbZ\} = \emptyset,    \lb{3.11} 
\end{equation}
it was proven by Fuchs \cite{Fu66} (cf.\ Hille \cite[Theorem~9.5.1]{Hi97}) that the perturbed Euler differential system \eqref{3.2} has fundamental matrix solutions of the form
\begin{equation}
Y(\zeta) = \sum_{m \in \bbN_0} C_m \, \zeta^{m I_N + A_0} C, \quad C_0 = I_N, \; 
C_{\ell} \in \bbC^{N \times N}, \; \ell \in \bbN,   \lb{3.12} 
\end{equation}
where again $C \in \bbC^{N \times N}$ is nonsingular. 

The case where the spectral assumption \eqref{3.11} on $A_0$ is violated is much more involved\footnote{In fact, we quote Hille \cite[p.~344]{Hi97} in this context: ``\dots A number of arguments are available in the literature all of them more or less corny. What I shall give here is not the corniest; \dots''}.  What follows is a shortened description of Hille \cite[Theorem~9.5.2]{Hi97}, a modified version of Frobenius' method: If \eqref{3.11} does not hold, fundamental matrix solutions of the perturbed Euler differential system \eqref{3.2} are of the form
\begin{equation}
Y(\zeta) = \sum_{j=0}^M [\ln(\zeta)]^j \sum_{m\in \bbN_0} C_{m,j} \, \zeta^{m I_N + A_0} C, 
\quad C_{0,0} = [M!] I_N, \; C_{m,j} \in \bbC^{N \times N},     \lb{3.13} 
\end{equation}
and once again $C \in \bbC^{N \times N}$ is nonsingular. A characterization of $M$ in \eqref{3.13} is possible, see, for instance, \cite[p.~342--352]{Hi97}. 

We conclude this overview by specializing the 1st-order $N \times N$ perturbed Euler system \eqref{3.2} to the $N$th-order scalar case (a special case of which is depicted in \eqref{3.1}). Consider the scalar $N$th-order differential equation 
\begin{equation}
y^{(N)}(\zeta) + b_{N-1}(\zeta) y^{(N-1)}(\zeta) + \cdots + b_1(\zeta) y'(\zeta) + b_0(\zeta) y(\zeta) = 0,  \lb{3.14}
\end{equation}
where the coefficients $b_j(\dott)$, $0 \leq j \leq N-1$, are of the form
\begin{equation}
b_j(\zeta) = \zeta^{j-N} a_j(\zeta), \quad a_j(\zeta) = \sum_{m \in \bbN_0} a_{j,m} \, \zeta^m,    \lb{3.15}
\end{equation} 
with $a_j(\dott)$ entire (resp., analytic in $D(0;R)$). The scalar ODE \eqref{3.14} subordinates to the perturbed Euler differential system \eqref{3.2} upon identifying $A(\zeta)$ with the $N \times N$ matrix 
\begin{equation}
\begin{pmatrix}
0 & 1 & 0 & 0 & \dots & & & 0 \\
0 & 1 & 1 & 0 & \dots & & & 0 \\
0 & 0 & 2 & 1 & \dots & & & 0 \\
0 & 0 & 0 & 3 & \dots & &  & 0 \\
\vdots & \vdots & \vdots & \vdots & \dots & & & \vdots \\
\vdots & \vdots & \vdots & \vdots & \dots & \ddots & \ddots & 0 \\
0 & 0 & 0 & 0 & \dots & & & 1 \\
- a_0(\zeta) & - a_1(\zeta) & - a_2(\zeta) & - a_3(\zeta) & \dots & & & (N-1) - a_{N-1}(\zeta)  
\end{pmatrix}      \lb{3.16} 
\end{equation}
and identifying $Y(\zeta)$ with $(Y_1(\zeta), \dots, Y_N(\zeta))$, where the solutions 
$Y_j(\dott) \in \bbC^{N \times 1}$ are given by 
\begin{equation}
Y_j(\dott) = (y_{j,1}(\dott), \dots, y_{j,N}(\dott))^\top, \quad y_{j,k}(\zeta) = \zeta^{k-1} y_j^{(k-1)}(\zeta), \quad 
1 \leq j,k \leq N,     \lb{3.17}
\end{equation}
with $y_j(\dott)$, $1 \leq j \leq N$, linearly independent solutions of \eqref{3.14}. In this scalar context the 
matrix $A_0 \in \bbC^{N \times N}$ in \eqref{3.3} is thus of the form
\begin{equation}
A_0 = \begin{pmatrix}
0 & 1 & 0 & 0 & \dots & & & 0 \\
0 & 1 & 1 & 0 & \dots & & & 0 \\
0 & 0 & 2 & 1 & \dots & & & 0 \\
0 & 0 & 0 & 3 & \dots & &  & 0 \\
\vdots & \vdots & \vdots & \vdots & \dots & & & \vdots \\
\vdots & \vdots & \vdots & \vdots & \dots & \ddots & \ddots & 0 \\
0 & 0 & 0 & 0 & \dots & & & 1 \\
- a_{0,0} & - a_{1,0} & - a_{2,0} & - a_{3,0} & \dots & & & (N-1) - a_{N-1,0}  
\end{pmatrix}      \lb{3.18} 
\end{equation}
and hence the eigenvalues $\alpha_q$ of $A_0$ prominently figuring in the solution \eqref{3.10} are determined via the indicial equation \eqref{3.9}, $D_N(\dott) = 0$, where 
\begin{align} 
D_N(z) &= {\det}_{\bbC^N} (z I_N - A_0)     \no \\
&= \sum_{k=0}^N a_{N-k,0} \begin{cases} \prod_{r=1}^{N-k} [z - (r-1)], & 0 \leq k \leq N-1, \\
1, & k=N,     
\end{cases}  \quad a_{N,0} = 1, \; z \in \bbC.     \lb{3.19}
\end{align}
 
Given these results we can return to the half-line differential expression $\tau_{2n}(c)$ in \eqref{3.1}, the special case of the scalar case \eqref{3.14} with $N = 2n$ and (frequently explicitly indicating the $c$-dependence of 
the coefficients) 
\begin{equation}
b_j(\zeta;c) = 0, \; 1 \leq j \leq 2n-1, \quad b_0(\zeta;c) = (-1)^n c \, \zeta^{-2n} - (-1)^n \mu, \quad \mu \in \bbC,   \lb{3.20} 
\end{equation}
equivalently, 
\begin{equation}
a_j(\zeta;c) = 0, \;, 1 \leq j \leq 2n-1, \quad a_0(\zeta;c) = (-1)^n c  - (-1)^n \mu \, \zeta^{2n}, \quad \mu \in \bbC.     \lb{3.21}
\end{equation} 
In this case the indicial equation further reduces to 
\begin{equation}
D_{2n}(z;c) = \prod_{j=1}^{2n} [z - (j-1)] + (-1)^n c = 0, \quad z \in \bbC.    \lb{3.22} 
\end{equation}

Thus, we can state the following result.

\begin{theorem} \lb{t3.2}
Let $c \in \bbR$, $\mu \in \bbC$. Then for any $d \in (0,\infty)$, the number of $L^2((0,d); dx)$-solutions of 
$\tau_{2n}(c) y(\mu,\dott) = \mu y(\mu;\dott)$, denoted by $\#_{L^2}\big(\tau_{2n}(c)|_{(0,d)}\big)$, is independent of $\mu$. In particular,
\begin{equation}
n \leq \#_{L^2}\big(\tau_{2n}(c)|_{(0,d)}\big) \leq 2n.     \lb{3.23}
\end{equation}
Moreover, the deficiency indices $n_{\pm}(T_{2n, min}(c))$ $($with $T_{2n, min}(c)$ representing the closure of 
$\tau_{2n}(c) \big|_{C_0^{\infty}((0,\infty))}$ in $L^2((0,\infty); dx)$$)$ equal
\begin{equation}
n_{\pm} (T_{2n, min}(c)) = \#_{L^2}\big(\tau_{2n}(c)|_{(0,d)}\big) - n.     \lb{3.24} 
\end{equation}
and hence 
\begin{equation}
0 \leq n_{\pm} (T_{2n, min}(c)) \leq n.     \lb{3.25} 
\end{equation} 
In particular, 
\begin{align}
\begin{split} 
& \text{$T_{2n, min}(c)$ is self-adjoint $\big($equivalently, $\ddot T_{2n, min}$ is essentially 
self-adjoint\,$\big)$}     \\
& \quad \text{in $L^2((0,\infty);dx)$ if and only if \ $\#_{L^2}\big(\tau_{2n}(c)|_{(0,d)}\big) = n$.}
\end{split}
\end{align}
\end{theorem}
\begin{proof}
The $\mu$-independence of $\#_{L^2}\big(\tau_{2n}(c)|_{(0,d)}\big)$ follows from the structure of the solutions $Y_j$ in \eqref{3.10}, the fact that for each $d \in (0,\infty)$, the power $x^{\alpha}$ lies in $L^2((0,d); dx)$ if and only if 
$\Re(\alpha) > - 1/2$, independently of the presence of any logarithmic factors, and finally that only the spectrum of $A_0$ determines the powers $\alpha_q$ in \eqref{3.10}. 

Since $c \in \bbR$, $\tau_{2n}(c)$ possesses an anti-unitary conjugation operator (effected by complex conjugation of elements in $L^2((0,\infty); dx)$) and one obtains by \eqref{2.17}, 
\begin{equation} 
n_+ (T_{2n, min}(c)) = n_- (T_{2n, min}(c)). 
\end{equation} 
Moreover by a special case of Kodaira's decomposition principle \eqref{2.20} for deficiency indices, 
\begin{align}
n_{\pm}(T_{2n, min}(c)) &= n_{\pm}\Big(\tau_{2n}(c)\big|_{C_0^{\infty}((0,d))}\Big) +
n_{\pm}\Big(\tau_{2n}(c)\big|_{C_0^{\infty}((d,\infty))}\Big) - 2n   \no \\
&= n_{\pm}\Big(\tau_{2n}(c)\big|_{C_0^{\infty}((0,d))}\Big) - n     \no \\
&= \#_{L^2}\big(\tau_{2n}(c)|_{(0,d)}\big) - n,     \lb{3.26} 
\end{align} 
since 
\begin{equation}
n_{\pm}\Big(\tau_{2n}(c)\big|_{C_0^{\infty}((d,\infty))}\Big) = n.     \lb{3.27} 
\end{equation}
Relation \eqref{3.27} holds since $\tau_{2n}(c)$ is regular at $d$ and, as  $x^{-2n}$ is bounded on the interval $[d,\infty)$ (cf.\ \cite[Sect.~14.7]{Na68}), $\tau_{2n}(c)$ is in the limit point case at $\infty$ since 
$(-1)^n d^{2n}/dx^{2n}$ is in the limit point case at $\infty$. Moreover, by \eqref{2.19}, 
\begin{equation}
n \leq n_{\pm}\Big(\tau_{2n}(c)\big|_{C_0^{\infty}((0,d))}\Big) \leq 2n,    \lb{3.28} 
\end{equation}
implying \eqref{3.23} and \eqref{3.25}. 
\end{proof}

\begin{remark} \lb{r3.3}
$(i)$ The independence of $\#_{L^2}\big(\tau_{2n}(c)|_{(0,d)}\big)$ with respect to $\mu$ permits one  to choose the by far simplest situation by taking $\mu = 0$ when counting the number of $L^2((0,d); dx)$-solutions of $\tau_{2n}(c) y(\mu,\dott) = \mu y(\mu;\dott)$. This in turn grants one to focus on solutions of the simple power-type $x^{\alpha}$ as in \eqref{3.10} (ignoring the possibility of additional logarithmic factors which, however, cannot influence the $L^2$- or non-$L^2$-behavior of solutions near $x=0$). In particular, considering
\begin{equation}
y_{\alpha}(x) = x^{\alpha} P(\ln(x)), \quad x \in (0,\infty), \; \alpha \in \bbC,     \lb{3.29}
\end{equation} 
where $P(\dott)$ is any polynomial, then for all $d \in (0,\infty)$,  
\begin{align}
y_{\alpha}(\dott) \in L^2((0,d); dx) \, \text{ if and ony if } \, \Re(\alpha) > -1/2.    \lb{3.30} 
\end{align}
Thus, by \eqref{3.10}, $\Re(\alpha) > - 1/2$, respectively, $\Re(\alpha) \leq - 1/2$, is the criterion deciding whether or not a particular solution with 
power-type behavior $x^{\alpha}$ (again, ignoring possible logarithmic factors) contributes to 
$\#_{L^2}\big(\tau_{2n}(c)|_{(0,d)}\big)$. \\[1mm]
$(ii)$ It will be shown in Corollary \ref{c4.8} that any permissible integer value for $\#(\tau_{2n}|_{(0,d)})$ in \eqref{3.23} actually is attained for some $c \in \bbR$.
\hfill $\diamond$ 
\end{remark}

\begin{remark} \lb{r3.3a}
One observes that $D_{2n}(.;c)$ possesses the symmetry
\begin{equation}
D_{2n}(-(1/2) + n + z) = D_{2n}(-(1/2) + n - z).
\end{equation} 
In particular, at $z=0$ one obtains
\begin{equation}
D_{2n}((-1/2) + n) = (-1)^n \Bigg( \prod_{j=1}^{n} [j - 1/2]^2 + c\Bigg) = (-1)^n \left(\f{[(2n -1)!!]^{2}}{2^{2n}} + c \right).
\end{equation} 
Consequently, for $c=-[(2n -1)!!]^{2}\big/2^{2n}$ one has a double zero at $\alpha=k - (1/2)$ and there are two solutions
of the type
\begin{equation}
y_1(0,x,c) = x^{k-(1/2)}, \qquad y_2(0,x,c) = x^{k-(1/2)} \ln(x)
\end{equation}
in this case.
\hfill $\diamond$ 
\end{remark}

Next, we now recall the special situation $n=1$ which is explicitly solvable for general spectral parameter $\mu$ in terms of Bessel functions as follows: 

\begin{example} \lb{e3.4} Assuming the case $n=1$ in \eqref{3.1} we consider 
\begin{align}
\begin{split} 
- y''(\mu,x) + c x^{-2} y(\mu,x) = \mu y(\mu,x),&    \\
\mu \in \bbC, \; x \in (0,\infty), \; c  \in \bbR.&     \lb{3.31} 
\end{split} 
\end{align}
The associated characteristic equation 
\begin{equation}
D_2(z; c) = z(z - 1) - c = 0, 
\lb{3.36} 
\end{equation}
has the following two complex-valued solutions
\begin{align} 
\begin{split} 
& \alpha_1(c) = (1/2) - \sqrt{c + (1/4)}, \\ 
& \alpha_2(c) =  (1/2) + \sqrt{c + (1/4)},      \lb{3.37} 
\end{split} 
\end{align}
choosing the principal branch for $[ \dott ]^{1/2}$ with branch cut $(-\infty,0]$, such that 
\begin{equation}
z^{1/2} = r^{1/2} e^{i \varphi/2}, \quad z = r e^{i \varphi}, \quad r, r^{1/2} \in [0,\infty), \; \varphi \in (-\pi,\pi]. 
\lb{3.32a}
\end{equation}
With this convention in place one checks that for all $c \in \bbR$, one has the ordering, 
\begin{equation}
\Re(\alpha_1(c)) \leq 1/2 \leq \Re(\alpha_2(c)).
\end{equation}

\noindent 
{\bf $(\alpha)$ Generic case:} Suppose $c\in \mathbb{R}$ is such that  
\begin{equation}
[\alpha_1(c)-\alpha_2(c)]/2 \not\in \mathbb{Z}.
\end{equation} 
Then the nonhomogenous differential equation \eqref{3.31} has the following fundamental system of solutions  
$($cf.\ \cite[No.~9.1.49, p.~362]{AS72}$)$
\begin{align}
y_1(\mu,x; c) &= (\pi/2) \mu^{-\gamma(c)/2}
x^{1/2} J_{\gamma(c)} \big(\mu^{1/2}x\big),    \no \\
y_2(\mu,x; c) &= \sin(\pi\gamma(c)) \mu^{\gamma(c)/2} x^{1/2} J_{-\gamma(c)} \big(\mu^{1/2}x\big),    \lb{3.32} \\
& \hspace*{3.25cm} \mu\in\bbC, \; x\in(0,\infty),   \no 
\end{align}
where
\begin{equation}
\gamma(c) = \sqrt{c + (1/4)}, \quad \gamma \in [0,\infty), \quad c \in \bbR, 
\end{equation} 
(Thus, $\gamma(c) \in \{[0,\infty) \backslash\bbN_0\} \cup i (0,\infty)$ in the generic case.) 

\medskip

\noindent
{\bf $(\beta)$ Exceptional Cases:} Suppose $c \in \bbR$ is such that 
\begin{equation} 
[\alpha_1(c)-\alpha_2(c)]/2 \in \mathbb{Z}, 
\end{equation}
then
\begin{equation}
c = k^2 - (1/4), \quad k \in \bbN_0.
\end{equation}
More precisely, for $k\in \bbN_0$,
\begin{equation}
[\alpha_{1}(c)-\alpha_{2}(c)]/2 = \pm k \, \text{ if and only if } \, c=k^2 - (1/4). 
\end{equation}
Furthermore,
\begin{equation}
\alpha_1(c)= \alpha_2(c) \, \text{ if and only if } \, c = - 1/4. 
\end{equation} 

In the exceptional case, where $\gamma(c) = k \in \bbN_0$, one obtains 
\begin{align}
y_1\big(\mu,x; k^2-(1/2)\big) &= (\pi/2) \mu^{-k/2} x^{1/2} J_k \big(\mu^{1/2}x\big), \no \\
y_2\big(\mu,x; k^2-(1/2)\big) &= \mu^{k/2} x^{1/2} \big[ -Y_k \big(\mu^{1/2}x\big) 
+ \pi^{-1}\ln(\mu)J_k \big(\mu^{1/2}x\big)\big],  \lb{3.33} \\
& \hspace*{1.25cm} \mu\in\bbC, \; x\in(0,\infty), \; c \in \big\{k^2-(1/4)\big\}_{k\in\bbN_0}. \no
\end{align}
Here $J_{\kappa}(\dott)$ represent the standard Bessel functions of order $\kappa \in \bbC$ and first kind, and $Y_k(\dott)$ denotes the Bessel function of order $k \in \bbN_0$ and second kind $($see, e.g., 
\cite[Ch.~9]{AS72}$)$. Moreover, one verifies $($cf.\ \cite[p.~360]{AS72}$)$ that 
\begin{equation}
W(y_2(\mu,\dott, c), y_1(\mu,\dott; c)) = 1, \quad
\mu \in\bbC, \; c \in \bbR      \lb{3.34}
\end{equation} 
$($here $W(f,g) = fg' - f'g$ denotes the Wronkian of $f$ and $g$$)$, 
and that the fundamental system of solutions $y_1(\mu,\cdot; c), y_2(\mu,\cdot, c)$ \eqref{3.32}, \eqref{3.33} of \eqref{3.31} is entire with respect to $\mu \in \bbC$ for fixed $x \in (0,\infty)$, and real-valued for $\mu\in\bbR$.

As $\mu \to 0$, the fundamental systems of solutions \eqref{3.32}, \eqref{3.33}, upon disregarding normalization, greatly simplify to
\begin{align} 
& y_1(0,x; c) = x^{\alpha_1(c)}, \quad c \in \bbR,  \quad 
y_2(0,x; c) = \begin{cases} x^{\alpha_2(c)}, 
\quad c \in \bbR \backslash \{- 1/4\}, \\
x^{1/2} \ln(x), \quad c = -1/4;  \end{cases}  \no \\
& \hspace*{10.15cm} x\in(0,\infty),    \lb{3.35}
\end{align} 
underscoring once again the advantage of choosing $\mu = 0$. 

One observes that in accordance with \eqref{1.9} (see also \eqref{1.10}) and Remark \ref{r3.3a}, the logarithmic case in \eqref{3.35} occurs at $c = -1/4$, that is, precisely at the borderline of semiboundedness of $T_{min,2}(c)$.

Thus, determining whether or not $\Re(\alpha_j(c) > - 1/2$, $j=1,2$, one concludes that 
\begin{equation}
 \#_{L^2}\big(\tau_{2}(c)|_{(0,d)}\big) = \begin{cases}
 1, & \mbox{if}\quad c \geq 3/4, \\
 2, & \mbox{if}\quad c < 3/4. 
 \end{cases}    \lb{3.38} 
\end{equation}
\end{example}

\begin{remark} \lb{r3.4a}
In view of the next example, where $n=2$, in fact, in view of the general case $n \in \bbN$, it might be interesting to rewrite the Bessel function solutions in the case $n=1$ in terms of the corresponding generalized hypergeometric function and Meijer's $G$-function as follows: In the generic case, where $c\in \mathbb{R}$ is such that $[\alpha_1(c)-\alpha_2(c)]/2 \not\in \mathbb{Z}$, the nonhomogenous differential equation \eqref{3.31} has the following fundamental system of solutions  
\begin{align}
y_1(\mu,x; c) &= x^{\alpha_1(c)} \,_{0} F_1\left(\!\begin{array}{c}\\
{\scriptstyle  1+\f{\alpha_1(c)-\alpha_2(c)}{2}}
\end{array} \bigg\vert\,  -\f{\mu x^2}{4} \right),   \no \\
y_2(\mu,x; c) &= x^{\alpha_2(c)} \,_{0} F_1\left(\!\begin{array}{c}\\
{\scriptstyle  1+\f{\alpha_2(c)-\alpha_1(c)}{2}}
\end{array} \bigg\vert\,  -\f{\mu x^2}{4} \right),    \lb{3.39} \\
& \hspace*{3.35cm} \mu\in\bbC, \; x\in(0,\infty).   \no 
\end{align}
Here $\,_{0} F_1\Big(\!\begin{array}{c}\\ {\scriptstyle  b_1} \end{array} \Big\vert\,  \dott \Big)$ 
represents the generalized hypergeometric function given by
\begin{align}
\,_{0} F_1\Big(\!\begin{array}{c}\\ {\scriptstyle  b_1} \end{array} \Big\vert\, \zeta \Big) = \sum_{k \in \bbN_0} \f{\zeta^k}{(b_1)_k k!}, \quad 
b_1 \in \bbC \backslash \{- \bbN_0\}, \; \zeta \in \bbC, 
\end{align}
with $(a)_k$ denoting Pochhammer's symbol,
\begin{equation}
(a)_0 =1, \quad (a)_k = \prod_{j=0}^{k-1} (a+j) = \Gamma(a+k)/\Gamma(a), \quad k \in \bbN, \; a \in \bbC.
\end{equation}
In particular, $\,_{0} F_1\Big(\!\begin{array}{c}\\ {\scriptstyle  b_1} \end{array} \Big\vert\, \zeta \Big)$ 
is entire in $\zeta \in \bbC$ and 
\begin{equation}
\,_{0} F_1\Big(\!\begin{array}{c}\\ {\scriptstyle  b_1} \end{array} \Big\vert\, \zeta \Big) 
\underset{\zeta \to 0}{=} 1 + \Oh(\zeta). 
\end{equation}

In the exceptional case, where $\gamma(c) = k \in \bbN_0$, one obtains 
\begin{align}
y_1\big(\mu,x; k^2-(1/2)\big) &= x^{k+(1/2)} \,_{0} F_1\left(\!\begin{array}{c}\\
{\scriptstyle  1+k} \end{array} \bigg\vert\,  -\f{\mu x^2}{4} \right),   \no \\
y_2\big(\mu,x; k^2-(1/2)\big) &= \Gamma(k+1) 2^k \mu^{-k/2} x^{1/2} 
G_{0,2}^{2,0}\left(
\!\!\begin{array}{c}\\ {\scriptstyle k/2 ; -k/2} \end{array} \bigg\vert\, - \f{\mu x^2}{4}\right)  \lb{3.52} \\
& \quad + \big[\pi (-1)^{k+1} i^{k+1} + \ln(\mu)\big] x^{k+(1/2)} 
\,_{0} F_1\left(\!\begin{array}{c}\\ {\scriptstyle  1+k}
\end{array} \bigg\vert\,  -\f{\mu x^2}{4} \right),   \no \\
& \hspace*{2.8cm} \mu\in\bbC, \; x\in(0,\infty), \; c \in \big\{k^2-(1/4)\big\}_{k\in\bbN_0}. \no
\end{align}
Here Meijer's $G$-function, $G_{0,2}^{2,0}\Big(\!\begin{array}{c}\\ {\scriptstyle  c_1, c_2} \end{array} \Big\vert\,  \dott \Big)$, is given by a Mellin--Barnes-type integral,
\begin{equation}
G_{0,2}^{2,0}\Big(\!\begin{array}{c}\\ {\scriptstyle  c_1, c_2} \end{array} \Big\vert\, \zeta \Big) 
= \f{1}{2\pi i} \int_{\cC} ds \, \zeta^s \Gamma(c_1-s) \Gamma(c_2-s), 
\lb{3.53} 
\end{equation}
where $\cC$ is a contour beginning and ending at $+\infty$ encircling all poles of $\Gamma(c_j-s)$, $j=1,2$, once in negative orientation, and the left-hand side of \eqref{3.53} is defined as the (absolutely convergent) sum of residues of the right-hand side. The exceptional case where $c_1$ and $c_2$ differ by an integer is treated by a limiting argument. (For more details see \cite{GH23}.) 
\hfill $\diamond$
\end{remark}

For details on generalized hypergeometric functions and Meijer's $G$-function we refer, for instance, to 
\cite{BS13}, \cite[Ch.~IV, Sects.~5.3--5.6]{EMOT53}, \cite[Ch.~V]{Lu69}, \cite[Ch.~V]{Lu75}, and \cite[Ch.~16]{OLBC10}, \cite[Sect.~8.2]{PBM90}. 

\begin{example} \lb{e3.5} Assuming the case $n=2$ in \eqref{3.1} we consider 
\begin{align}
\begin{split} 
y''''(\mu,x) + c x^{-4} y(\mu,x) = \mu y(\mu,x),&   \lb{3.42} \\
x \in (0,\infty), \; \mu \in \bbC, \; c \in \bbR.&
\end{split} 
\end{align}
The associated characteristic equation 
\begin{equation}
D_4(z; c) = z(z - 1)(z-2)(z-3) - c = 0,  \quad z \in \bbC, \; c \in \bbR,   \lb{3.56} 
\end{equation}
has the following four complex-valued solutions,
\begin{align}
\begin{split} 
\alpha_1(c)&= \left[3-\sqrt{5+4 \sqrt{1-c}} \,\right]\bigg/2,\\
\alpha_2(c)&= \left[3-\sqrt{5-4 \sqrt{1-c}} \,\right]\bigg/2,\\
\alpha_3(c)&= \left[3+\sqrt{5-4 \sqrt{1-c}} \,\right]\bigg/2,\\
\alpha_4(c)&= \left[3+\sqrt{5+4 \sqrt{1-c}} \, \right]\bigg/2; \quad c \in \bbR,  \lb{3.57} 
\end{split}
\end{align}
employing the principal branch \eqref{3.32a} for $[\dott]^{1/2}$. 
With this convention, one checks that for all $c\in \mathbb{R}$, one has
\begin{equation}
\Re(\alpha_1(c)) \leq \Re(\alpha_2(c))  \leq 3/2 \leq  \Re(\alpha_3(c)) \leq \Re(\alpha_4(c)).
\end{equation} 

\noindent 
{\bf $(\alpha)$ Generic case:} Suppose $c\in \mathbb{R}$ is such that  
\begin{equation} 
[\alpha_{j}(c)-\alpha_{j'}(c)]/4 \not\in \mathbb{Z}, \, \text{ for all } \, 1 \leq j, j' \leq4, \; j\not=j'.
\end{equation} 
Then the nonhomogenous differential equation \eqref{3.42} has the following fundamental system of solutions, 
\begin{align}
\begin{split}
y_1(\mu,x;c)&=x^{\alpha_1(c)} \,_{0} F_3\left(\!\begin{array}{c}\\
{\scriptstyle  1+\f{\alpha_1(c)-\alpha_2(c)}{4}, 1+\f{\alpha_1(c)-\alpha_3(c)}{4}, 1+\f{\alpha_1(c)-\alpha_4(c)}{4}}
\end{array} \bigg\vert\,  \f{\mu x^4}{256} \right),\\
y_2(\mu,x;c)&=x^{\alpha_2(c)} \,_{0} F_3\left(\!\begin{array}{c}\\
{\scriptstyle  1+\f{\alpha_2(c)-\alpha_1(c)}{4}, 1+\f{\alpha_2(c)-\alpha_3(c)}{4}, 1+\f{\alpha_2(c)-\alpha_4(c)}{4}}
\end{array} \bigg\vert\,  \f{\mu x^4}{256} \right),\\
y_3(\mu,x;c)&=x^{\alpha_3(c)} \,_{0} F_3\left(\!\begin{array}{c}\\
{\scriptstyle  1+\f{\alpha_3(c)-\alpha_1(c)}{4}, 1+\f{\alpha_3(c)-\alpha_2(c)}{4}, 1+\f{\alpha_3(c)-\alpha_4(c)}{4}}
\end{array} \bigg\vert\,  \f{\mu x^4}{256} \right),\\
y_4(\mu,x;c)&=x^{\alpha_4(c)} \,_{0} F_3\left(\!\begin{array}{c}\\
{\scriptstyle  1+\f{\alpha_4(c)-\alpha_1(c)}{4}, 1+\f{\alpha_4(c)-\alpha_2(c)}{4}, 1+\f{\alpha_4(c)-\alpha_3(c)}{4}}
\end{array} \bigg\vert\,  \f{\mu x^4}{256} \right);
\end{split} \\
& \hspace*{7cm} \mu \in \bbC, \; x \in (0,\infty).   \no 
\end{align}
Asymptotically, 
\begin{equation} 
y_j(\mu,x;c) \underset{x\downarrow 0}{=} x^{\alpha_j(c)}[1+\Oh(x)], \quad 1 \leq j \leq 4,
\end{equation} 
and thus, the four functions are indeed linearly independent.

Here $\,_{0} F_3\Big(\!\begin{array}{c}\\ {\scriptstyle  b_1, b_2, b_3} \end{array} \Big\vert\,  \dott \Big)$ 
represents the generalized hypergeometric function given by
\begin{align}
\,_{0} F_3\Big(\!\begin{array}{c}\\ {\scriptstyle  b_1, b_2, b_3} \end{array} \Big\vert\, \zeta \Big) = \sum_{k \in \bbN_0} \f{\zeta^k}{(b_1)_k (b_2)_k (b_3)_k k!}, \quad 
b_1, b_2, b_3 \in \bbC \backslash \{- \bbN_0\}, \; \zeta \in \bbC. 
\end{align}
Again, $\,_{0} F_3\Big(\!\begin{array}{c}\\ {\scriptstyle  b_1, b_2, b_3} \end{array} \Big\vert\, \zeta \Big)$ is entire in $\zeta \in \bbC$ and 
\begin{equation}
\,_{0} F_3\Big(\!\begin{array}{c}\\ {\scriptstyle  b_1, b_2, b_3} \end{array} \Big\vert\, \zeta \Big) 
\underset{\zeta \to 0}{=} 1 + \Oh(\zeta). 
\end{equation}

That these functions are in fact solutions of \eqref{3.42} can be confirmed by direct verification using the differential equation for generalized hypergeometric functions. 

\medskip

\noindent
{\bf $(\beta)$ Exceptional Cases:} Suppose $c \in \bbR$ is such that 
\begin{equation} 
[\alpha_{j}(c)-\alpha_{j'}(c)]/4 \in \mathbb{Z} \, \text{ for some } \, 1 \leq j, j' \leq 4, \; j\not=j', 
\end{equation}
then
\begin{equation}
\text{either } \, c = 1 - 20 k^2 + 64 k^4, \, \text{ or, } \, c = - (9/16) +10 k^2 - 16 k^4, \quad k \in \bbN_0.
\end{equation}
More precisely, for $k\in \bbN_0$,
\begin{align}
\begin{split}
[\alpha_{1}(c)-\alpha_{2}(c)]/4 &= \pm k \, \text{ implies } \, c=1 - 20 k^2 + 64 k^4,\\
[\alpha_{1}(c)-\alpha_{3}(c)]/4 &= \pm k \, \text{ implies } \, c=1 - 20 k^2 + 64 k^4,\\
(\alpha_{1}(c)-\alpha_{4}(c)]/4 &= \pm k  \, \text{ implies } \, c= - (9/16) +10 k^2 - 16 k^4,\\
[\alpha_{2}(c)-\alpha_{3}(c)]/4 &= \pm k \, \text{ implies } \, c= - (9/16) +10 k^2 - 16 k^4,\\
(\alpha_{2}(c)-\alpha_{4}(c)]/4 &= \pm k \, \text{ implies } \, c=1 - 20 k^2 + 64 k^4,\\
[\alpha_{3}(c)-\alpha_{4}(c)]/4 &= \pm k \, \text{ implies } \, c=1 - 20 k^2 + 64 k^4.
\end{split} 
\end{align}
Furthermore,
\begin{equation}
\alpha_1(c)= \alpha_2(c) \, \text{ if and only if } \, \alpha_3(c)= \alpha_4(c)
 \, \text{ if and only if } \, c =1
\end{equation} 
and 
\begin{equation} 
\alpha_2(c)= \alpha_3(c) \, \text{ if and only if } \, c = - 9/16.
\end{equation}  

\medskip
\noindent
If $c=1$, then 
\begin{equation}
\alpha_1(1) = \alpha_2(1) = \big[3 - \sqrt{5}\,\big]\big/2, \quad 
\alpha_3(1) = \alpha_4(1) = \big[3 + \sqrt{5}\,\big]\big/2,
\end{equation} 
and a fundamental system of solutions is given by, 
\begin{align}
\begin{split} 
y_1(\mu,x;1)&=x^{[3-\sqrt{5}]/2} \,_{0} F_3\left(\!\!\begin{array}{c}\\
{\scriptstyle  1, 1-\f{\sqrt{5}}{4},1-\f{\sqrt{5}}{4}}
\end{array} \bigg\vert\,  \f{\mu x^4}{256} \right),\\
y_2(\mu,x;1)&=G_{0,4}^{2,0}\left(
\!\!\begin{array}{c}\\
{\scriptstyle  \f{3-\sqrt{5}}{8} ,\f{3-\sqrt{5}}{8} ; \f{3+\sqrt{5}}{8},\f{3+\sqrt{5}}{8}} 
\end{array}
\bigg\vert\,  \f{\mu x^4}{256}
\right),\\
y_3(\mu,x;1)&=x^{[3+\sqrt{5}]/2} \,_{0} F_3\left(\!\!\begin{array}{c}\\
{\scriptstyle  1, 1+\f{\sqrt{5}}{4},1+\f{\sqrt{5}}{4}}
\end{array} \bigg\vert\,  \f{\mu x^4}{256} \right),\\
y_4(\mu,x;1)&=G_{0,4}^{2,0}\left(
\!\!\begin{array}{c}\\
{\scriptstyle  \f{3+\sqrt{5}}{8} ,\f{3+\sqrt{5}}{8} ; \f{3-\sqrt{5}}{8},\f{3-\sqrt{5}}{8}} 
\end{array}
\bigg\vert\,  \f{\mu x^4}{256}
\right);
\end{split} \\
& \hspace*{3.3cm} \mu \in \bbC, \; x \in (0,\infty).   \no 
\end{align}
Asymptotically, 
\begin{align}
\begin{split} 
& y_2(\mu,x;1) \underset{x \downarrow 0}{=} c_2 x^{[3-\sqrt{5}]/2} \ln(x)[1 + \Oh(x)], \\ 
& y_4(\mu,x;1) \underset{x \downarrow 0}{=} c_4 x^{[3+\sqrt{5}]/2} \ln(x)[1 + \Oh(x)].
\end{split}
\end{align}
Here Meijer's $G$-function, $G_{0,4}^{2,0}\Big(\!\begin{array}{c}\\ {\scriptstyle  c_1, c_2; c_3, c_4} \end{array} \Big\vert\,  \dott \Big)$, is again given by a Mellin--Barnes-type integral,
\begin{equation}
G_{0,4}^{2,0}\Big(\!\begin{array}{c}\\ {\scriptstyle  c_1, c_2; c_3, c_4} \end{array} \Big\vert\, \zeta \Big) 
= \f{1}{2\pi i} \int_{\cC} ds \, \zeta^s \f{\Gamma(c_1-s) \Gamma(c_2-s)}{\Gamma(1-c_3+s) \Gamma(1-c_4+s)}, 
\lb{3.67} 
\end{equation}
where $\cC$ is a contour beginning and ending at $+\infty$ encircling all poles of $\Gamma(c_j-\cdot)$, $j=1,2$, once in negative orientation, and the left-hand side of \eqref{3.67} is defined as the (absolutely convergent) sum of residues of the right-hand side. The exceptional case where $c_1$ and $c_2$ differ by an integer is once more treated by a limiting argument. 

\medskip

\noindent
If $c = 1 - 20 k^2 + 64 k^4$, $k\in \bbN$, then 
\begin{align}
\begin{split}
& \alpha_1\big(1 - 20 k^2 + 64 k^4\big) = \Big[3 - 4k - \sqrt{5 - 16k^2} \,\Big]\Big/2,    \\
&  \alpha_2\big(1 - 20 k^2 + 64 k^4\big) = \Big[3 - 4k + \sqrt{5 - 16k^2} \,\Big]\Big/2,    \\
& \alpha_3\big(1 - 20 k^2 + 64 k^4\big) = \Big[3 + 4k - \sqrt{5 - 16k^2} \,\Big]\Big/2,    \\
& \alpha_4\big(1 - 20 k^2 + 64 k^4\big) = \Big[3 + 4k + \sqrt{5 - 16k^2} \,\Big]\Big/2,    \\
\end{split}
\end{align}
and a fundamental system of solutions is given by,
\begin{align}
& y_1\big(\mu,x; 1 - 20 k^2 + 64 k^4\big)     \no \\
& \quad =G_{0,4}^{2,0}\left(
\!\!\begin{array}{c}\\
{\scriptstyle   \f{3-4 k-\sqrt{5-16 k^2}}{8} , \f{3+4 k-\sqrt{5-16 k^2}}{8} ; \f{3-4 k+\sqrt{5-16 k^2}}{8} , \f{3+4 k+\sqrt{5-16 k^2}}{8}} 
\end{array}
\bigg\vert\,  \f{\mu x^4}{256}\right),     \no \\
& y_2\big(\mu,x; 1 - 20 k^2 + 64 k^4\big)     \no \\
& \quad =G_{0,4}^{2,0}\left(
\!\!\begin{array}{c}\\
{\scriptstyle   \f{3-4 k+\sqrt{5-16 k^2}}{8} , \f{3+4 k+\sqrt{5-16 k^2}}{8} ; \f{3-4 k-\sqrt{5-16 k^2}}{8} , \f{3+4 k-\sqrt{5-16 k^2}}{8}} 
\end{array}
\bigg\vert\,  \f{\mu x^4}{256}\right),     \no \\
& y_3\big(\mu,x; 1 - 20 k^2 + 64 k^4\big)      \\
& \quad = x^{[(3+4k) - \sqrt{5-16k^2}\,]/2} \,_{0} F_3\left(\!\!\begin{array}{c}\\
{\scriptstyle 1 + k , 1 + k - \f{\sqrt{5 - 16 k^2}}{4},1 - \f{\sqrt{5 - 16 k^2}}{4}}
\end{array} \bigg\vert\,  \f{\mu x^4}{256} \right),     \no \\
& y_4\big(\mu,x; 1 - 20 k^2 + 64 k^4\big)     \no \\
& \quad = x^{[(3+4k) + \sqrt{5-16k^2}\,]/2} \,_{0} F_3\left(\!\!\begin{array}{c}\\
{\scriptstyle 1+k, 1 + k + \f{\sqrt{5 - 16 k^2}}{4},  1 + \f{\sqrt{5 - 16 k^2}}{4}}
\end{array} \bigg\vert\,  \f{\mu x^4}{256} \right);    \no \\
& \hspace*{7.85cm} \mu \in \bbC, \; x \in (0,\infty).   \no
\end{align}
Asymptotically,
\begin{align}
\begin{split} 
& y_1\big(\mu,x; 1 - 20 k^2 + 64 k^4\big) \underset{x \downarrow 0}{=} 
x^{[(3-4k) - \sqrt{5-16k^2}\,]/2} \ln(x) [1 + \Oh(x)],       \\
& y_2\big(\mu,x; 1 - 20 k^2 + 64 k^4\big) \underset{x \downarrow 0}{=} 
x^{[(3-4k) + \sqrt{5-16k^2}\,]/2} \ln(x) [1 + \Oh(x)].
\end{split} 
\end{align} 
\medskip
\noindent
If $c=-9/16$, then 
\begin{align}
\begin{split} 
& \alpha_1(-9/16) = \big[3 - \sqrt{10}\,\big]\big/2,    \\ 
& \alpha_2(-9/16) =\alpha_3(-9/16) = 3/2,      \\ 
& \alpha_4(-9/16) = \big[3 + \sqrt{10}\,\big]\big/2,
\end{split} 
\end{align}
and a fundamental system of solutions is given by, 
\begin{align}
\begin{split} 
y_1(\mu,x;-9/16)&=x^{[3-\sqrt{10}]/2} \,_{0} F_3\left(\!\!\begin{array}{c}\\
{\scriptstyle  1-\f{\sqrt{10}}{4}, 1-\f{\sqrt{10}}{8},1-\f{\sqrt{10}}{8}}
\end{array} \bigg\vert\,  \f{\mu x^4}{256} \right),\\
y_2(\mu,x;-9/16)&=x^{3/2} \,_{0} F_3\left(\!\!\begin{array}{c}\\
{\scriptstyle  1, 1-\f{\sqrt{10}}{8},1+\f{\sqrt{10}}{8}}
\end{array} \bigg\vert\,  \f{\mu x^4}{256} \right),\\
y_3(\mu,x;-9/16)&=G_{0,4}^{2,0}\left(
\!\!\begin{array}{c}\\
{\scriptstyle  \f{3}{8} ,\f{3}{8} ; \f{3-\sqrt{10}}{8},\f{3+\sqrt{10}}{8}} 
\end{array}
\bigg\vert\,  \f{\mu x^4}{256}
\right),\\
y_4(\mu,x;-9/16)&=x^{[3+\sqrt{10}]/2} \,_{0} F_3\left(\!\!\begin{array}{c}\\
{\scriptstyle  1+\f{\sqrt{10}}{4}, 1+\f{\sqrt{10}}{8},1+\f{\sqrt{10}}{8}}
\end{array} \bigg\vert\,  \f{\mu x^4}{256} \right);
\end{split} \\
& \hspace*{4.6cm} \mu \in \bbC, \; x \in (0,\infty).   \no 
\end{align}
Asymptotically, 
\begin{equation}  
y_3(\mu,x, -9/16) \underset{x \downarrow 0}{=} c_3 x^{3/2} \ln(x) [1 + \Oh(x)].
\end{equation} 

One observes that the case $c = -9/16$, is again precisely the borderline of semiboundedness of 
$T_{min,4}(c)$ again in accordance with \eqref{1.9} (see also \eqref{1.10}) and Remark \ref{r3.3a}.

\medskip 
\noindent
If $c=- (9/16) +10 k^2 - 16 k^4$, $k\in \bbN$, then 
\begin{align}
\begin{split}
& \alpha_1\big(-(9/16) +10 k^2 - 16 k^4\big) = (3-4k)/2,   \\
& \alpha_2\big(-(9/16) +10 k^2 - 16 k^4\big) = \Big[3 - \sqrt{10-16k^2} \,\Big]\Big/2, \\
& \alpha_3\big(-(9/16) +10 k^2 - 16 k^4\big) = \Big[3 + \sqrt{10-16k^2} \,\Big]\Big/2, \\
& \alpha_4\big(-(9/16) +10 k^2 - 16 k^4\big) = (3+4k)/2,
\end{split}
\end{align}
and a fundamental system of solutions is given by, 
\begin{align}
& y_1\big(\mu,x;-(9/16) +10 k^2 - 16 k^4\big)     \no \\
& \quad =G_{0,4}^{2,0}\left(
\!\!\begin{array}{c}\\
{\scriptstyle   \f{3-4k}{8} , \f{3+4 k}{8} ; \f{3-\sqrt{10-16 k^2}}{8} , \f{3+\sqrt{10-16 k^2}}{8}} 
\end{array}
\bigg\vert\,  \f{\mu x^4}{256} \right),   \no \\
& y_2\big(\mu,x;-(9/16) +10 k^2 - 16 k^4\big)      \no \\
& \quad =x^{[3 - \sqrt{10-16k^2}\,]/2} 
\,_{0} F_3\left(\!\!\begin{array}{c}\\
{\scriptstyle  \f{8-2\sqrt{10 - 16 k^2}}{8}, \f{8-4k-\sqrt{10 - 16 k^2}}{8}, \f{8+4k-\sqrt{10 - 16 k^2}}{8}}
\end{array} \bigg\vert\,  \f{\mu x^4}{256} \right),   \no \\
& y_3\big(\mu,x;-(9/16) +10 k^2 - 16 k^4\big)      \\
& \quad =x^{[3 + \sqrt{10-16k^2}\,]/2} 
\,_{0} F_3\left(\!\!\begin{array}{c}\\
{\scriptstyle  \f{8+2\sqrt{10 - 16 k^2}}{8}, \f{8-4k+\sqrt{10 - 16 k^2}}{8}, \f{8+4k+\sqrt{10 - 16 k^2}}{8}}
\end{array} \bigg\vert\,  \f{\mu x^4}{256} \right),     \no \\
& y_4\big(\mu,x;-(9/16) +10 k^2 - 16 k^4\big)      \no \\
& \quad =x^{(3+4k)/2} 
\,_{0} F_3\left(\!\!\begin{array}{c}\\
{\scriptstyle 1+k, \f{8+4k-\sqrt{10 - 16 k^2}}{8},  \f{8+4k+\sqrt{10 - 16 k^2}}{8}}
\end{array} \bigg\vert\,  \f{\mu x^4}{256} \right);      \no \\
& \hspace*{6.85cm} \mu \in \bbC, \; x \in (0,\infty).   \no
\end{align}
Asymptotically,   
\begin{equation} 
y_1(x) \underset{x \downarrow 0}{=} c_1 x^{(3-4k)/2} [1 + \Oh(x)] + c_2 x^{(3+4k)/2} \ln(x) [1 + \Oh(x)].
\end{equation} 

Once more, as $\mu \to 0$, the fundamental system of solutions of \eqref{3.42} considerably simplifies  to 
\begin{align}
\begin{split} 
& y_1(0,x;c) = x^{\alpha_1(c)}, \quad y_2(0,x;c) = x^{\alpha_2(c)},   \\
&y_3(0,x;c) = x^{\alpha_3(c)}, \quad y_4(0,x;c) = x^{\alpha_4(c)}; \quad c \in \bbR \backslash \{1, - 9/16\}, 
\end{split} \\
\begin{split} 
& y_1(0,x;1) = x^{[3-\sqrt{5}]/2}, \quad y_2(0,x;1) = x^{[3-\sqrt{5}]/2} \ln(x),    \\
& y_3(0,x;1) = x^{[3+\sqrt{5}]/2}, \quad y_4(0,x;1) = x^{[3+\sqrt{5}]/2} \ln(x),  \quad c=1,    
\end{split} \\
\begin{split} 
& y_1(0,x;-9/16) = x^{[3-\sqrt{10}]/2}, \quad y_3(0,x;-9/16) = x^{3/2},   \\
&  y_3(0,x;-9/16) = x^{3/2} \ln(x), \quad y_4(0,x;-9/16) = x^{[3+\sqrt{10}]/2}, \quad c=-9/16;  
\end{split} \\
& \hspace*{10.35cm} x \in (0,\infty).   \no
\end{align} 
By inspection, one verifies that $\tau_4(c) y_j(0,\dott;c) =0$, $1 \leq j \leq 4$. Alternatively, one can apply the theory of $n$th-order Euler differential equations as presented, for instance, in \cite[p.~122--123]{CL85}. 

Thus, determining whether or not $\Re(\alpha_j(c) > - 1/2$, $1 \leq j \leq 4$, one concludes that 
\begin{equation}
 \#_{L^2}\big(\tau_{4}(c)|_{(0,d)}\big) = \begin{cases}
 2, & \mbox{if}\quad c \geq 45, \\
 4, & \mbox{if}\quad - (7!!)/2^4 \leq c < 45, \\
 3, & \mbox{if}\quad c < - (7!!)/2^4.
 \end{cases}    \lb{3.58} 
\end{equation}
(Explicitly, $(7!!)/2^4 = 105/16$.)
\end{example}

Without going into further details we note that also the higher-order examples $n \in \bbN$, $n \geq 3$, can be explicitly solved in terms generalized hypergeometric functions and Meijer's $G$-function (this will be  discussed in \cite{GH23}).

\section{On the Real Part of the Roots of $D_{2n}(\dott;c)$, $c\in \bbR$}\lb{s4}

For $n\in \bbN$ and $c\in \bbR$, let $D_{2n}(\dott;c)$ be the polynomial given by \eqref{3.22} and note that all of its coefficients are real.
The goal of this section is to   determine how many  of the roots of $D_{2n}(\dott;c)$
have real part $> - 1/2$. Results of this sort are typically approached by using the Routh--Hurwitz criterion. We propose a  different  approach here, even though Hurwitz's ideas still play a central role.

Let us begin by fixing some notation. For $c\in \bbR$, let the roots of $D_{2n}(\dott;c)=0$ be denoted $\alpha_j(c)$, $j=1,\ldots,2n$.
By the continuous dependence of the roots of a polynomial on the coefficients (see 
\cite[Theorem~(1.4)]{Ma66}), we may choose our labelling such that  each $\alpha_j(c)$ is a continuous function 
of $c$ and
\begin{equation} \lb{4.1}
\Re (\alpha_1(c)) \leq  \Re (\alpha_2(c)) \leq \cdots  \leq \Re (\alpha_n(c)) 
 \leq \cdots  \leq \Re (\alpha_{2n}(c)),    
\quad c\in \bbR.
\end{equation}
Note that $\Re (\alpha_j(0)) =\alpha_j(0) =j-1$ for $j=1,\ldots,2n$. The fact that  
 \begin{equation}\label{4.1a}
\mbox{ $D_{2n}(\dott;0)$ has $2n$ distinct \emph{real} roots $>- 1/2$} 
\end{equation} 
will be of crucial importance in all that follows.
 
\begin{example} \lb{e4.1} 
Figure~\ref{fig: graph} shows the  graphs of the the real parts of the roots of $D_6(\dott;c)$ as  functions of $c\in \mathbb{R}$.
The scale for the $x$-axis has been chosen such that $x=c^{1/6}$ for $c>0$ and $x=\sgn(c)|c|^{1/6}$ for $c<0$.
The dotted lines show the graphs of the real parts of the roots of $(\dott)^6 - c=0$ as functions of $c$.
One notes that these dotted lines are straight lines precisely because of our special choice of scale for the $x$-axis.
Furthermore, as $c\to \pm \infty$, the graph of each function $\Re (\alpha_j(c))$ approaches one of these straight lines asymptotically. One observes that for $c\ll 0$, one has 
$\Re(\alpha_1(c))=\Re(\alpha_2(c))<
\Re(\alpha_3(c))=\Re(\alpha_4(c))<
\Re(\alpha_5(c))=\Re(\alpha_6(c))$.
Similarly, for $c\gg 0$, one infers that 
$\Re(\alpha_1(c))<\Re(\alpha_2(c))=
\Re(\alpha_3(c))<\Re(\alpha_4(c))=
\Re(\alpha_5(c))<\Re(\alpha_6(c))$.

As will be shown later, we have
\begin{align}
\begin{split} 
\Re(\alpha_1(c)) \leq -\f{1}{2}&\quad \mbox{iff}  \quad c\leq \f{2240 \left(214-7 \sqrt{1009}\right)}{27} \approx -693.0  \\ 
& \qquad \quad \mbox{or} \quad c\geq \f{10395}{64}\approx 162.4, \\[10pt]
\Re(\alpha_2(c)) \leq -\f{1}{2}&\quad \mbox{iff}  \quad c\leq \f{2240 \left(214-7 \sqrt{1009}\right)}{27} \approx -693.0     \\ 
& \qquad \quad \mbox{or} \quad c\geq \f{2240 \left(214+7 \sqrt{1009}\right)}{27}\approx 36201.2, 
\\[10pt]
\Re(\alpha_3(c)) \leq -\f{1}{2}&\quad \mbox{iff}  \quad c\geq \f{2240 \left(214+7 \sqrt{1009}\right)}{27}\approx 36201.2,
\end{split} 
\end{align}
where the algebraic numbers on the right are roots of the quadratic equation $27c^2 - 958720 c - 677376000 = 0$. If $j\in \{4,5,6\}$, then  $\Re(\alpha_j(c)) > - 1/2$ for all $c\in \bbR$. 
\end{example}

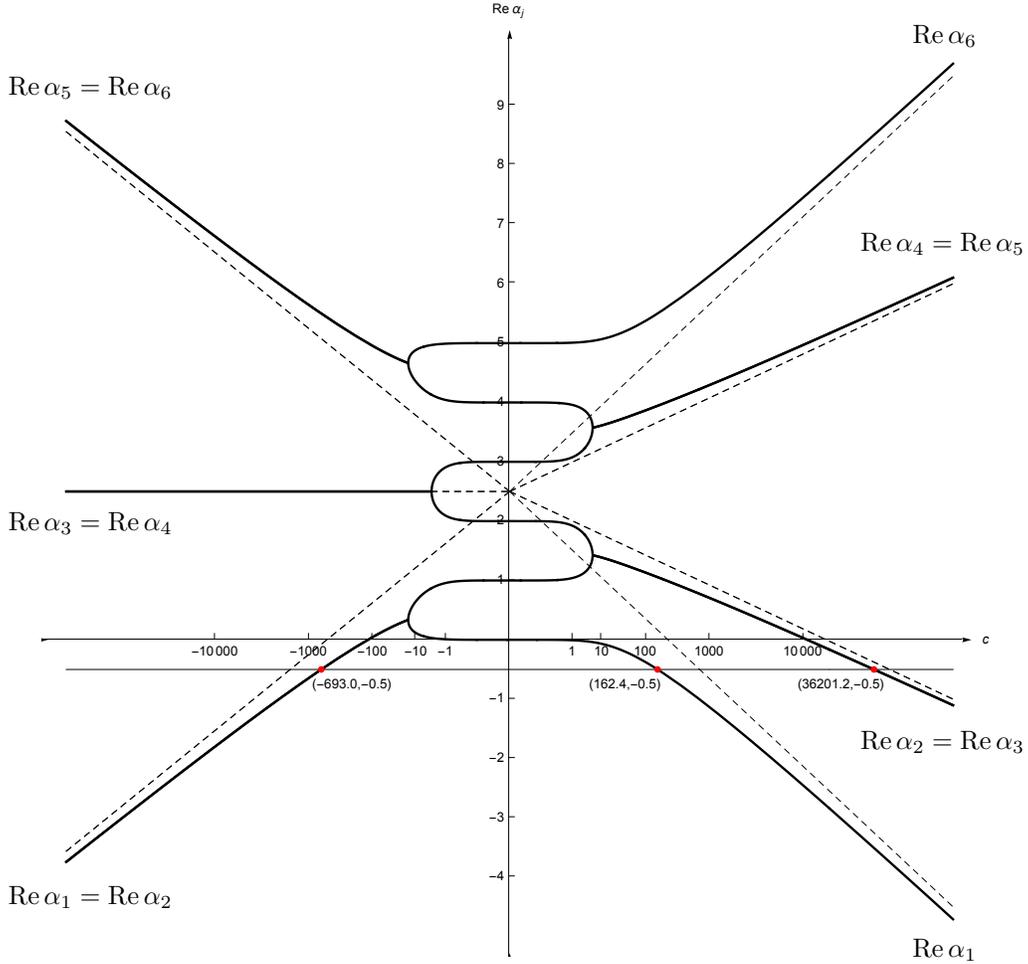
\begin{figure}
\begin{overpic}[scale=.77]{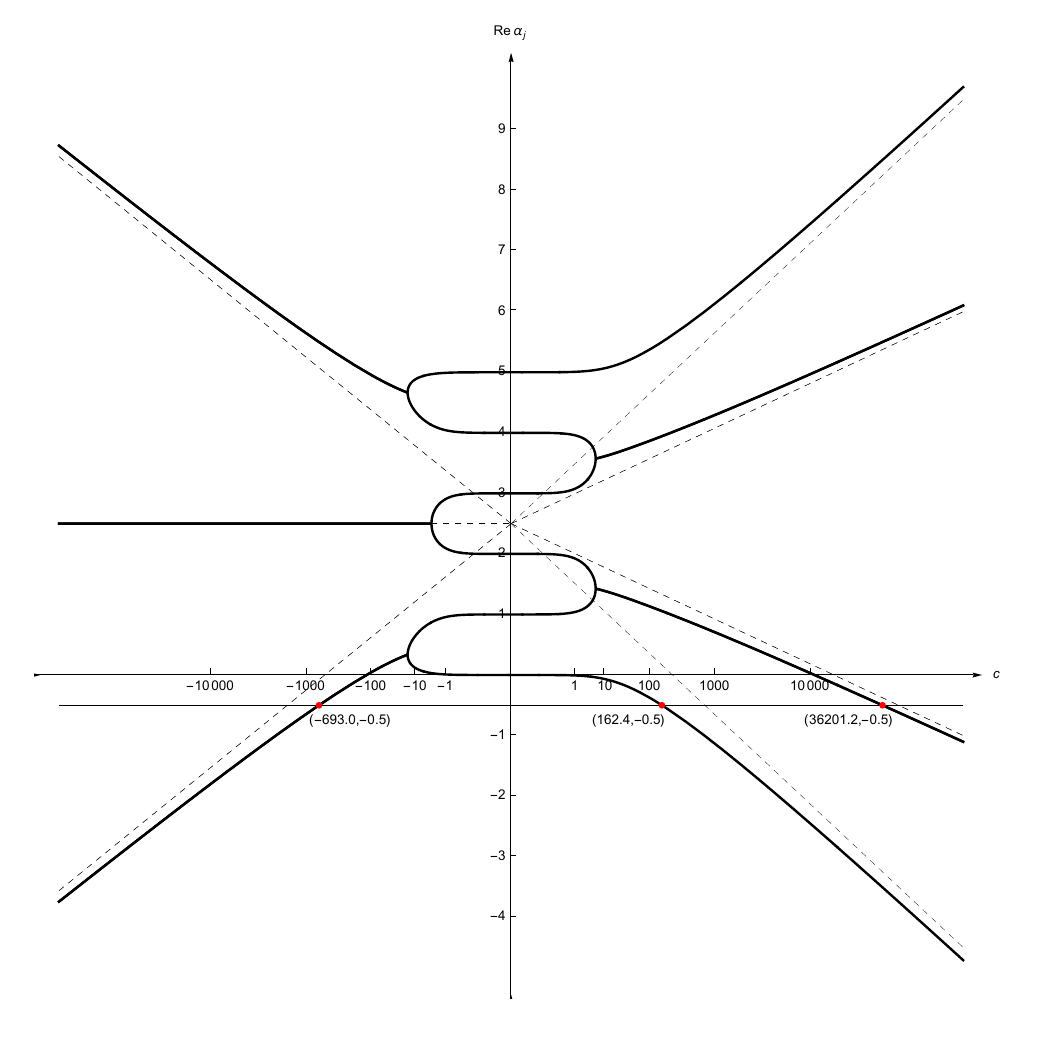}
\put(0,9){\small$\Re \alpha_1=\Re\alpha_2$}
\put(0,45){\small$\Re\alpha_3 =\Re\alpha_4 $}
\put(0,87){\small$\Re\alpha_5 =\Re\alpha_6 $}
\put(87,4){\small$\Re\alpha_1$}
\put(82,24){\small$\Re\alpha_2=\Re\alpha_3 $}
\put(82,72){\small$\Re\alpha_4=\Re\alpha_5 $}
\put(87,92){\small$\Re\alpha_6$}
\end{overpic}
\caption{Graphs of the the real parts of the roots of $D_6(\dott;c)$ as  functions of $c\in \mathbb{R}$.}
 \lb{fig: graph}
\end{figure}

 The proof of our main result, Theorem~\ref{t4.6}, concerning the real parts of the roots of $D_{2n}(\dott;c)$, $c\in \mathbb{R}$, will depend on three lemmas.
 The first lemma states that for any $c\in \mathbb{R}$, the polynomial $D_{2n}(\dott;c)$ cannot have more than two roots (counting multiplicity)
 having the same real part. More precisely, we have the following result: 

\begin{lemma}\lb{l4.2}
For $j,j'\in \{1,2,\ldots,2n\}$ and $c\in \bbR$,
\begin{equation} \lb{4.3}
\Re (\alpha_j(c)) =\Re (\alpha_{j'}(c)) \, \text{ implies } \, |j-j'|\leq 1,
\end{equation}
Furthermore, if  $\Re (\alpha_j(c)) =\Re (\alpha_{j'}(c))$ and $|j-j'|=1$, then $\alpha_j(c) ,\alpha_{j'}(c)\not\in \mathbb{R}$ and
$\overline{\alpha_j(c)}= \alpha_{j'}(c)$.
\end{lemma}
\begin{proof}
Let $c\in  \mathbb{R}$ and note that
\begin{equation}\lb{4.4}
\f{d}{dz} D_{2n}(z;c) = \f{d}{dz} \left(D_{2n}(z;0)+(-1)^nc\right)=
\f{d}{dz} D_{2n}(z;0), \quad z\in \mathbb{C}.
\end{equation}
 By \eqref{4.1a}
\begin{equation} 
\text{all of the roots of the derivative of $D_{2n}(\dott;0)$ are real and simple,}   \lb{4.5a} 
\end{equation}
it follows that $D_{2n}(\dott;c)$ does not have real roots of multiplicity greater than two. Moreover, since $c\in\mathbb{R}$, all roots of $D_{2n}(\dott;c)$ are real or complex conjugates. Arguing by contradiction, suppose the polynomial $D_{2n}(\dott;c)$ has more than two roots (counting multiplicity) having the same real part. Then
\begin{align} \lb{cl4.3}
\begin{split}
&\mbox{there exist two roots $z_1,z_2\in \bbC$  of  $D_{2n}(\dott;c)$ such that}\\ 
&\quad \mbox{$\Re (z_1) =\Re(z_2)$ and $0 \leq \Im(z_1)<\Im(z_2)$.}
\end{split}
\end{align}
We now use the Grace--Heawood theorem to obtain a contradiction. More precisely, 
we use the following corollary of (the proof of) the Grace--Heawood theorem, which is stated on page 126 of \cite{RS02} as a ``Supplement'':
\begin{center}
\begin{minipage}{.9\textwidth}
\emph{If $z_1,z_2\in \bbC$ are two distinct roots of a complex polynomial of degree $\geq 2$, then neither of the two closed half-planes whose boundary is the perpendicular bisector of the line segment $[z_1,z_2]$ is devoid of any critical points of the polynomial.}
\end{minipage}
\end{center}
When applied to the two roots $z_1,z_2$ of $D_{2n}(\dott;c)$
as in the claim, this leads to a contradiction as follows.
Note that the perpendicular bisector of the line segment $[z_1,z_2]$ in our situation is of the form 
$\{z\in \bbC\mid \Im(z) = y_0\}$, where $y_0:=[\Im(z_1)+\Im(z_2)]/2>0$.
Now recall that by \eqref{4.5a} all the critical points of $D_{2n}(\dott;c)$
are real. Thus, the closed half-plane $\{z\in \bbC \,| \Im(z)\geq y_0 \}$ would be devoid of any critical points of $D_{2n}(\dott;c)$. This is the desired contradiction.
\end{proof}

The second lemma is concerned with the asymptotic behavior of the real parts of the  roots of 
$D_{2n}(\dott;c)$ as $c\rightarrow \pm\infty$.

\begin{lemma}\lb{l4.4} 
For $j\in \{1,2,\ldots,2n\}$ and $c\in \bbR$,
\begin{equation}\lb{4.5}
 \lim_{c\to + \infty} \Re (\alpha_j(c)) =
\begin{cases}
-\infty, & 1\leq j \leq n, \\ 
+ \infty, & n+1\leq j \leq 2n, 
\end{cases}
\end{equation}
and 
\begin{equation} \lb{4.6}
 \lim_{c\to -\infty} \Re (\alpha_j(c)) = 
\begin{cases}
-\infty, & 1\leq j \leq n-1,  \\
\textstyle{n-(1/2)},  & n \leq j\leq n+1,  \\
+ \infty, &n+2\leq  j \leq 2n.
\end{cases}
\end{equation}
\end{lemma}
\begin{proof}
For the purpose of this proof, let $f(\cdot)$ be the polynomial given by
\begin{equation}
f(z):=D_{2n}(z+(n-(1/2));0),\quad z\in \bbC.
\end{equation}
The half-integer $n-(1/2)$ is the center of mass of the roots of $D_{2n}(\dott;0)$ and hence 
the center of mass of the roots of $f(\dott)$ is $0$.
In other words, 
\begin{equation}\lb{4.8}
\mbox{if we write $f(z)=\sum_{j=0}^{2n}a_j z^j$, then  $a_{2n-1}=0$.}
\end{equation}
For $z_0\in \bbC$, it will be convenient to define polynomials
$f(\dott;z_0)$ and $g(\dott;z_0)$ by
\begin{equation}
f(z;z_0):=f(z)-z_0^{2n}, \quad 
g(z;z_0):=z^{2n}-z_0^{2n}, \quad z\in \bbC.
\end{equation}
One notes that if $z_0^{2n}=(-1)^{n-1}c$, then $f(z;z_0)=D_{2n}(z-(1/2);c)$ for all $z\in \bbC$. 

Next, let $\varepsilon >0$.
We claim that there exists 
a real number $R >0$ such that if $|z_0|>r$, then the polynomial $f(\cdot;z_0)$ has a unique  root in  the open disc $U(z_0;\varepsilon):=\{z\in \mathbb{C} \,|\, |z-z_0|<\varepsilon\}$. Notice that   $g(\cdot;z_0)$ has a unique root in $U(z_0;\varepsilon)$, namely $z_0$, as long as $|z_0|$ is sufficiently large. Thus, one can use Rouch\'e's theorem  as follows.
 Let $M:=\max\{|a_{2n-2}|,\ldots, |a_1|,|a_{0}|\}$. If $|z_0|\geq 1+\varepsilon$ and  $z\in \partial U(z_0;\varepsilon)$, then $1\leq  |z|\leq |z_0|+\varepsilon$
and hence (keeping in mind \eqref{4.8})
\begin{align}
\begin{split} 
|f(z;z_0)-g(z;z_0)|&=|a_{2n-2}z^{n-2}+\ldots+a_1z+a_0|\\
& \leq |a_{2n-2}| |z|^{2n-2}+ \cdots+ |a_1| |z|+|a_0| \\
&\leq M(|z|^{2n-2}+\cdots+ |z|+1) \\
& \leq (2n-1)M |z|^{2n-2}\\
& \leq (2n-1)M(|z_0|+\varepsilon)^{2n-2}.
\end{split} 
\end{align}
Furthermore, if  $|z_0|\geq 1+ \varepsilon$, then the minimum of $|g(\cdot;z_0)|$ on the boundary  $\partial U(z_0;\varepsilon)$ is attained at $z=(|z_0|- \varepsilon) z_0/|z_0|$ and hence for every 
$z\in \partial U(z_0;\varepsilon)$ one has 
\begin{align}
\begin{split}
|g(z;z_0)|&=|z^{2n}-z_0^{2n}|\geq |(|z_0|-\varepsilon)^{2n}-|z_0|^{2n}|\\
&=\varepsilon\, |(|z_0|-\varepsilon)^{2n-1}+\cdots+ (|z_0|-\varepsilon)+1|.
\end{split} 
\end{align}
One notes that if $|z_0|$ is sufficiently large, then 
\begin{equation} 
\varepsilon\, [(|z_0|-\varepsilon)^{2n-1}+\cdots+ (|z_0|-\varepsilon)+1] > (2n-1)M(|z_0|+\varepsilon)^{2n-2}
\end{equation} 
since the left-hand side is a polynomial in $|z_0|$ of degree $2n-1$ (with positive leading coefficient) and the right-hand side 
is a polynomial in $|z_0|$ of degree $2n-2$ (with positive leading coefficient.)
Therefore, if $|z_0|$ is sufficiently large, then 
\begin{equation} 
|g(z;z_0)| >  |f(z;z_0)-g(z;z_0)|\quad \mbox{for every $z\in \partial U(z_0;\varepsilon)$}
\end{equation} 
and hence, by Rouch\'e's theorem,   $f(\dott;z_0)$ and $g(\dott;z_0)$ have the same number of roots (counted with multiplicity) in  $U(z_0;\varepsilon)$.
It follows that there exists some $R>0$ such that if $|z_0|>R$, then $f(\dott;z_0)$ has a unique  root in  the open disc $U(z_0;\varepsilon)$. 

We can now complete the proof of Lemma \ref{l4.4}. For $c\in \bbR$, let the roots of   
\begin{equation}\lb{4.14}
[z-(n-(1/2))]^{2n}+(-1)^nc=0, \quad z\in \bbC,
\end{equation}
be denoted $\beta_j(c)$, $j=1,\ldots,2n$. One can choose a labelling such that  
\begin{equation}\lb{4.15}
\Re (\beta_1(c)) \leq  \Re (\beta_2(c)) \leq \cdots  \leq \Re (\beta_n(c)) 
 \leq \cdots  \leq \Re (\beta_{2n}(c)),    
\quad c\in \bbR.
\end{equation}
There is a statement  analogous to 
Lemma~\ref{l4.2}
for the roots $\beta_{j}(c)$, $j=1,\ldots,2n$. 
In light of this, there is a ``canonical'' labeling
for both the roots $\alpha_{j}(c)$ and 
$\beta_{j}(c)$ such that 
if $1\leq j<2n$ and $\Re(\alpha_j(c))= 
\Re (\alpha_{j+1}(c))$\quad [resp. $\Re(\beta_j(c))= 
\Re (\beta_{j+1}(c))$],
then $\Im (\alpha_j(c))<\Im (\alpha_{j+1}(c))$\quad [resp. $\Im (\beta_j(c))<\Im (\beta_{j+1}(c))$].
The roots of \eqref{4.14} are trivial to determine and a straightforward (but somewhat tedious) analysis shows that the asymptotic behavior of   $\Re(\beta_j(c))$
as $c\to \pm \infty$
is given by \eqref{4.5} and 
\eqref{4.6}, respectively, with $\alpha_j{(c)}$ replaced by $\beta_j{(c)}$, \quad $j=1,2\ldots, 2n$.

Now for $\varepsilon>0$ and $|c|\gg 0$, by the Rouch\'e argument from above applied to $z_0=\beta_j(c)$,
\begin{equation}
|\alpha_{j}(c) - \beta_{j}(c)|<\varepsilon, \quad j=1,2\ldots, 2n.
\end{equation}
Therefore, the asymptotic behavior of   $\Re(\beta_j(c))$
as $c\to \pm \infty$
is given by \eqref{4.5} and 
\eqref{4.6}, respectively.
\end{proof}

Finally, the last lemma is  related to the Routh--Hurwitz criterion, adapted to our situation. This takes some preparation. For $c\in  \mathbb{R}$, one first expands $D_{2n}(z-(1/2);c)$ as a polynomial in $z$,
\begin{equation}\lb{4.17}
D_{2n}( z -(1/2);c)=q_{2n} z^{2n}+ q_{2n-1} z^{2n-1} + \cdots +   q_1 z+ \big[q_0+(-1)^n c\big],
\end{equation}
and then considers the associated  $(2n \times 2n)$ Hurwitz matrix,
\begin{equation}\lb{4.18}
H_{2n}(c):=
\begin{pmatrix}
q_{2n-1}\!& \! q_{2n-3}\! &\! q_{2n-5}\! &\ \ \cdots& \ \  0\  \ &\ \ 0 \ \ & \ \ 0 \ \ &\\
q_{2n} & q_{2n-2} & q_{2n-4} &\ \  \ddots          & \vdots & \vdots & \vdots &\\
0 & q_{2n-1} & q_{2n-3} &\  \ \ddots        & \vdots & \vdots & \vdots &\\
\vdots & q_{2n} & q_{2n-2}& &               0     &\vdots  &\vdots &\\
\vdots & 0& q_{2n-1} & &                    \!\!\! \!\! \!\! q_0\!+\!(-1)^n c\!\! \!\! \!\!&\vdots & \vdots&\\
\vdots & \vdots & q_{2n} & &                    q_1  & 0  & \vdots &\\
\vdots  & \vdots  &0&    &                 q_2  &\!\! \!\! \!\! q_0\! +\! (-1)^n c\!\! \!\! \!\! &  \vdots&\\
\vdots  & \vdots   & \vdots &           & q_3  & q_1 &0 & \\[5pt]
0 & 0 & 0 &\ \ \cdots  &  q_4 &          q_2 &\!\! \!\! \!\! q_0\!+\!(-1)^n c\!\! \!\! \!\! &\\ 
\end{pmatrix}.
\end{equation}
One notes that $q_j\in \mathbb{Q}$ for all $j\in \{0,1,\ldots, 2n\}$. Furthermore, observe that $c$ only occurs in the even rows.
This implies that the function $\det\left(H_{2n}(\dott)\right)$ is a polynomial  of degree $n$ with rational coefficients.
By Laplace expansion along the last column,
\begin{equation}\lb{4.19}
\det\left(H_{2n}(c)\right)= \left[q_0+(-1)^n c\right]h_{n-1}(c),
\end{equation}
where $h_{n-1}(\dott)$ is a polynomial of degree $n-1$ with rational coefficients.
There is a simple closed expression for $q_0$, which is reminiscent of the expression on the 
right-hand side of \eqref{1.9}:
\begin{equation}\lb{4.20}
q_0=\f{(4n-1)!!}{2^{2n}}.
\end{equation}
Formula~\eqref{4.20} is easily proved by induction using that 
\begin{equation} 
q_0=D(-1/2;0)=\prod_{j=1}^{2n} [j-(1/2)].
\end{equation} 

\begin{lemma} \lb{l4.5} 
For  $j\in \{1,2,\ldots,2n\}$ and $c\in \bbR$, if 
$\Re(\alpha_{j}(c))=- 1/2$, then 
\begin{equation} 
\det(H_{2n}(c))=0, 
\end{equation} 
that is,
\begin{equation} 
c=(-1)^{n-1} q_0, \, \text{ or, } \, h_{n-1}(c)=0, 
\end{equation} 
where $h_{n-1}(\dott)$ is given by \eqref{4.19}.
\end{lemma}
\begin{proof}
Note that the roots of the polynomial \eqref{4.17} are just the roots of $D_{2n}(\dott; c)$ shifted by $1/2$, that is, roots of the polynomial \eqref{4.17} are $\alpha_{j}(c)+(1/2)$, where $j\in \{1,2,\ldots,2n\}$. 
It then follows from Orlando's formula (see \cite[\S~XV.7]{Ga59}) that 
\begin{equation}\lb{4.24}
h_{n-1}(c)= \prod_{\phantom{m}1\leq j_1< j_2\leq 2n} \{[\alpha_{j_1}(c) +(1/2)]+[\alpha_{j_2}(c) +(1/2)]\}. 
\end{equation}
Next, let  $j\in \{1,2,\ldots,2n\}$ and $c\in \bbR$
such that $\Re(\alpha_{j}(c))=-1/2$. First suppose $\alpha_{j}(c)\in \bbR$. Then $\alpha_{j}(c)=-1/2$ and 
$D_{2n}(-1/2; c) = D_{2n}(-1/2; 0)+(-1)^nc=0$, which implies that $c=(-1)^{n-1}q_0$. Next suppose $\alpha_{j}(c)\not\in \bbR$.
By Lemma~\ref{l4.2}, there exists some $j'\in \{1,2,\ldots,2n\}$, $j\not=j'$, such that 
$\alpha_{j'}(c)=\overline{\alpha_j(c)}$. Then
$[\alpha_{j}(c) + (1/2)] + [\alpha_{j'}(c) + (1/2)] = 0$ and 
hence $h_{n-1}(c)=0$ by \eqref{4.24}.
\end{proof}

We now have all the necessary ingredients to prove the main result of this section, Theorem \ref{t4.6}. In this context we will use the floor and ceiling notation: One recalls that for $n\in \bbN$, $\lfloor n/2 \rfloor$ denotes the greatest integer less than or equal to $n/2$; similarly, $\lceil n/2\rceil$ denotes the least integer greater than or equal $n/2$. Thus, for $n\in \bbN$, one has 
\begin{equation}\lb{4.25}
\lceil n/2\rceil =
\begin{cases}
\lfloor n/2\rfloor + 1 =(n+1)/2 & \mbox{if $n$ is odd},\\[5pt]
\lfloor n/2\rfloor =n/2 & \mbox{if $n$ is even}.
\end{cases}
\end{equation}
Recalling Remark \ref{r3.3}\,$(i)$, one obtains for $c \in \bbR$, $d \in (0,\infty)$, 
\begin{equation}
\#\left(\tau_{2n}(c)|_{(0,d)}\right) =
\mbox{the number of $j\in \{1,2,\ldots,2n\}$ such that
$\Re(\alpha_j(c))> - 1/2$.}
\end{equation}

\begin{theorem} \lb{t4.6}
$(i)$ For every $n\in \mathbb{N}$, $n\geq 2$, there exist $n$ real constants
\begin{equation}\lb{4.27}
c_{n}^{(1)}<c_{n}^{(2)}< \cdots < c_{n}^{(n)}
\end{equation}
such that the following items $(a)$--$(c)$ hold:\\[1mm]
\begin{itemize}
\item[$(a)$] For $c\in \mathbb{R}$, $d \in (0,\infty)$, one has 
\begin{equation}
\#\left(\tau_{2n}(c)|_{(0,d)}\right)=
\begin{cases}
n,& \mbox{if}\quad c \geq c_{n}^{(n)},\\[5pt]
n+2(n-k),& \mbox{if}\quad c_{n}^{(k)}\leq c<c_{n}^{(k+1)} \ \mbox{and}\  \lfloor n/2 \rfloor < k\leq n-1,\\[5pt]
2n, & \mbox{if}\quad c_{n}^{(k)}< c<c_{n}^{(k+1)}
 \ \mbox{and}\  k=\lfloor n/2 \rfloor,\\[5pt]
n+2k+1,& \mbox{if}\quad  c_{n}^{(k)}< c \leq c_{n}^{(k+1)}  \ \mbox{and}\ 1\leq k< \lfloor n/2 \rfloor ,\\[5pt]
n+1,& \mbox{if}\quad c\leq  c_{n}^{(1)}.\\[5pt]
\end{cases}     \lb{4.28}
\end{equation}
\item[$(b)$]   The constant $c_{n}^{(\lceil n/2\rceil)}$  is  given by the formula
\begin{equation}\lb{4.29}
c_{n}^{(\lceil n/2\rceil)} = (-1)^{n-1} \f{(4n-1)!!}{2^{2n}}.
\end{equation}
 
\item[$(c)$] The constants $c_{n}^{(1)},c_{n}^{(2)},\ldots c_{n}^{(\lceil n/2 \rceil-1)},c_{n}^{(\lceil n/2 \rceil+1)}, \ldots , c_{n}^{(n)}$ are the roots of the polynomial $h_{n-1}(\dott)$
of degree $n-1$ with rational coefficients. In addition,
\begin{equation} 
c_n^{(n)} \geq \f{(4n-1)!!}{2^{2n}} \underset{n \to \infty}{=} 2^{1/2} (2/e)^n n^{2n}[1 + \Oh(1/n)].   \lb{4.30} 
\end{equation}
\end{itemize}
$(ii)$ For $n=1$ one obtains  
\begin{equation}
\#\left(\tau_{2}(c)|_{(0,d)}\right) = \begin{cases} 1, & \mbox{if}\quad c \geq 3/4, \\
2, & \mbox{if}\quad c < 3/4. 
\end{cases} 
\end{equation}
\end{theorem}
\begin{proof}  $(i)$ The constants $c_{n}^{(1)},\ldots, c_{n}^{(n)}$
will turn out to be the roots of the  polynomial $\det(H_{2n}(\dott))$ of degree $n$ given by \eqref{4.18}.
However, it is not clear, a priori, that $\det(H_{2n}(\dott))$ has $n$ distinct real roots.
For that reason, we will have to define our  constants differently.

Next, we recall that the polynomial $D_{2n}(\dott;0)$ has $2n$ distinct  real roots, namely the non-negative integers $\alpha_{j}(0)=j-1$, where $j\in\{1,2\ldots, 2n\}$. 
In particular,  $\Re(\alpha_{j}(0))>- 1/2$
for all $j\in\{1,2\ldots, 2n\}$.
By Lemma~\ref{l4.4}, if $1\leq j\leq n-1$, one has  $\lim_{c\to -\infty}\Re(\alpha_{j}(c))=-\infty$ and hence 
$\{c < 0 \,| \Re(\alpha_{j}(c))=- 1/2\}$ is nonempty by continuity;
similarly,  if  $1\leq j\leq n$, then $\lim_{c\to \infty}\Re(\alpha_{j}(c))=-\infty$ and hence  
$\{c > 0 \,| \Re(\alpha_{j}(c))=- 1/2\}$ is nonempty by continuity.
Now, for  $1\leq k\leq n$, define
\begin{equation}
c_{n}^{(k)}:=
\begin{cases}
\min\{c \in \bbR \,| \Re(\alpha_{n - 2k+1}(c))=- 1/2\} & \mbox{if $1\leq k\leq \lfloor n/2\rfloor$}\\[5pt]
\max\{c \in \bbR \,| \Re(\alpha_{2 (k - \lfloor n/2\rfloor) - 1 }(c))=- 1/2\} & \mbox{if $\lfloor n/2\rfloor< k\leq n$.}
\end{cases}
\end{equation}
One notes that if $1\leq k\leq \lfloor n/2\rfloor$, then $1\leq n - 2k+1 \leq n-1$
and $c_n^{(k)}<0$;
similarly, if  $\lfloor n/2\rfloor < k\leq n$, then $1\leq 2 (k - \lfloor n/2\rfloor) - 1\leq n$ and $c_n^{(k)}>0$. By \eqref{4.1}, we then obtain
\begin{equation}\lb{4.32}
c_{n}^{(1)}\leq c_{n}^{(2)}\leq  \cdots \leq c_{n}^{(\lfloor n/2\rfloor)} < 0< 
c_{n}^{(\lfloor n/2\rfloor+1)}\leq \cdots
\leq c_{n}^{(n-1)} \leq c_{n}^{(n)}
\end{equation}
Next we use Lemma~\ref{l4.2} to show that all the inequalities in 
\eqref{4.32} are strict. Suppose 
$c_n^{(k)}=c_n^{(k+1)}$ for some $1\leq k \leq \lfloor n/2\rfloor-1$. Then  $\Re(\alpha_{n-2k+1}(c_n^{(k)}))=\Re(\alpha_{n-2k-1}(c_n^{(k)}))$ and since $|(n-2k+1)-(n-2k-1)|=2>1$, this  contradicts \eqref{4.3}. The same argument also yields a contradiction 
if $c_n^{(k)}=c_n^{(k+1)}$ for some $\lfloor n/2\rfloor<  k   \leq n-1$.
Therefore,  all the inequalities in 
\eqref{4.32} are strict.

We can say a bit more about  
the constants $c_{n}^{(\lfloor n/2\rfloor)}$
and $c_{n}^{(\lfloor n/2\rfloor+1)}$.
We claim that
\begin{equation}\lb{4.33}
c_{n}^{(\lfloor n/2\rfloor)} \leq  -q_0 < 0< q_0 \leq
c_{n}^{(\lfloor n/2\rfloor+1)},
\end{equation}
where $q_0=D_{2n}(-1/2;0)$ as before.
In fact, by \cite[Theorem~1(a)]{TG71}, the equation $D_{2n}({\textstyle z-(1/2)};c)=0$ has no purely imaginary solutions  $z\in i\bbR$ if $|c|< q_0$,
which implies \eqref{4.33}. Combining \eqref{4.32} and \eqref{4.33} this also implies
\begin{equation}
c_n^{(n)} \geq q_0 = \f{(4n-1)!!}{2^{2n}} = \f{\Gamma(4n)}{2^{4n-1} \Gamma(2n)}.     \lb{4.35}
\end{equation} 
Stirling's formula (see, e.g., \cite[No.~6.1.37]{AS72}),
\begin{equation}
\Gamma(z) \underset{\substack{z \to \infty \\ |\arg(z)|< \pi}}{=} (2 \pi)^{1/2} e^{-z} z^{z - (1/2)} [1 + \Oh(1/z)],
\end{equation}
then yields \eqref{4.30}.

By Lemma~\ref{l4.5}, $\det\big(H_{2n}\big(c_n^{(k)}\big)\big)=0$ for every $1\leq k\leq n$.
Since the constants $c_n^{(k)}$ 
are distinct and since 
$\det(H_{2n}(\dott))$ is a polynomial of degree $n$,
the polynomial $\det(H_{2n}(\dott))$ does not have any other roots. Furthermore, one of the constants $c_n^{(k)}$ must by equal to $(-1)^{n-1}q_0$ and the other $n-1$ constants must be the roots of the polynomial $h_{n-1}(\dott)$. If $n$ is odd, then
$(-1)^{n-1}q_0=q_0>0$ and it follows from \eqref{4.33} and \eqref{4.32} that $c_{n}^{(\lfloor n/2\rfloor+1)}=q_0$;
similarly, if $n$ is even, then
$(-1)^{n-1}q_0=-q_0<0$ and it follows from \eqref{4.33} and \eqref{4.32} that $c_{n}^{(\lfloor n/2\rfloor)}=-q_0$.
In either case, in light of \eqref{4.25}, we have $c_{n}^{(\lceil n/2\rceil)} =(-1)^{n-1}q_0$. Thus, recalling the formula for $q_0$ from \eqref{4.20}, we obtain \eqref{4.29}. This completes the proof of parts (b) and (c) of Theorem \ref{t4.6}. 

Before we prove part (a), we recall that by the continuity argument given in the first paragraph
of this proof, for every $1\leq j\leq  n-1$, there exists some $c<0$ such that $\Re(\alpha_j(c))=-1/2$. By our observations above, this $c$ must be one of the constants $c_{n}^{(k)}$ with 
$1\leq k\leq \lfloor n/2\rfloor$. Similarly, 
for every $1\leq j\leq  n$, there exists some $c>0$ such that $\Re(\alpha_j(c))=-1/2$ and, by our observations above, this $c$ must be one of the constants $c_{n}^{(k)}$ with $\lfloor n/2\rfloor +1 \leq k\leq n$. 

We will now prove part (a) in the case when $n$ is odd.
Then $n-1$ is even and $n-1=2\lfloor n/2\rfloor=2(n-\lceil n/2 \rceil)$.
By Lemma~\ref{l4.2}
and since $n-1=2\lfloor n/2\rfloor$,  
for every $1\leq k\leq \lfloor n/2\rfloor$, there are exactly two distinct $j,j'\in \{1,2\ldots, n-1\}$ such that $\Re\big(\alpha_j\big(c_{n}^{(k)}\big)\big)=\Re\big(\alpha_{j'}\big(c_{n}^{(k)}\big)\big)$. Furthermore, 
$c_{n}^{(\lfloor n/2\rfloor+1)}=
c_{n}^{(\lceil n/2\rceil)}=q_0$ and 
$\alpha_{1}\big(c_{n}^{(\lceil n/2\rceil)}\big)=- 1/2\in \bbR$.
By Lemma~\ref{l4.2} and since 
$n-1=2(n-\lceil n/2 \rceil)$,  for every $\lceil n/2\rceil+1\leq k\leq n$ , there are  exactly two distinct $j,j'\in \{2,3\ldots, n\}$ such that $\Re\big(\alpha_j\big(c_{n}^{(k)}\big)\big)=\Re\big(\alpha_{j'}\big(c_{n}^{(k)}\big)\big)$. The resulting situation is summarized in Figure~\ref{fig: constants}(a). We now use Figure~\ref{fig: constants}(a)  to understand how the value of $\#\left(\tau_{2n}(c)|_{(0,d)}\right)$ changes with $c\in \bbR$. For $c\leq c_{n}^{(1)}$, Figure~\ref{fig: constants}(b) shows that $\Re(\alpha_{j}(c)) > - 1/2$ if and only if  $n\leq j \leq 2n$. Therefore, $\#\left(\tau_{2n}(c)|_{(0,d)}\right)=n+1$ for $c\leq c_{n}^{(1)}$. As $c$ increases beyond $c_{n}^{(1)}$, the value of $\#\left(\tau_{2n}(c)|_{(0,d)}\right)$
jumps from $n+1$ to $n+3$ since for $c_{n}^{(1)}<c\leq c_{n}^{(2)}$,
$\Re(\alpha_{j}(c)) > - 1/2$ if and only if  $n-2\leq j \leq 2n$ (assuming that $n\geq 3$).
As $c$ increases more, the value of $\#\left(\tau_{2n}(c)|_{(0,d)}\right)$ increases by $2$ each time $c$ crosses one of the constants $c_{n}^{(k)}$ until $c$ reaches $c_{n}^{(\lfloor n/2 \rfloor)})$, when the value $\#\left(\tau_{2n}(c)|_{(0,d)}\right)$ only increases by $1$ from 
$2n-1$ to $2n$. From then on, 
the value of $\#\left(\tau_{2n}(c)|_{(0,d)}\right)$ starts decreasing by $2$ each time $c$ moves beyond one of the constants $c_{n}^{(k)}$ until, finally, $c$ passes $c_{n}^{(n)}$, and we have $\#\left(\tau_{2n}(c)|_{(0,d)}\right)=n$ since for  $c\geq  c_{n}^{(n)}$,  $\Re(\alpha_{j}(c)) > - 1/2$ if and only if  $n+1\leq j \leq 2n$. The result is the piecewise-formula  for $\#\left(\tau_{2n}(c)|_{(0,d)}\right)$
stated in part (a).

In the case when $n$ is even, the argument 
is, mutatis mutandis, the same. The situation is summarized in Figure~\ref{fig: constants}(b). 
The result is the same piecewise-formula  for $\#\left(\tau_{2n}(c)|_{(0,d)}\right)$
stated in part~(a).

$(ii)$ This has been discussed in Example \ref{e3.4}. 
\end{proof}

\begin{figure}[h]
\centering
\begin{subfigure}{0.90\textwidth}
\renewcommand\captionlabelfont{}
 \begin{tikzpicture} 
 [baseline,remember picture,scale=0.80,
    implies/.style={double,double distance=.8pt,-implies},
    dot/.style={shape=circle,fill=black,minimum size=2pt,
                inner sep=0pt,outer sep=2pt},
]
  \draw[implies]        (.5,-.25)   -- (1,.25) node [below=7pt,text centered,midway]{\tiny\begin{tabular}{l}$\Re(\alpha_{n-1}(c))$\\ $\Re(\alpha_{n-2}(c))$\end{tabular}};
  \node[circle,fill=black,inner sep=0pt,minimum size=3pt] at (0.75,0) {};
  \node at (2.3,0) [below=10pt] {$\cdots$};
  \draw[implies]        (3.5,-.25)   -- (4,.25) node [below=7pt,text centered,midway]{\tiny\begin{tabular}{l}$\Re(\alpha_{2}(c))$\\ $\Re(\alpha_{1}(c))$\end{tabular}};
  \node[circle,fill=black,inner sep=0pt,minimum size=3pt] at (3.75,0) {};
  \draw[->]        (7.5,.25)   -- (8,-.25) node [below=8pt,text centered,midway]{\tiny $\Re(\alpha_{1}(c))$\ };
  \node[circle,fill=black,inner sep=0pt,minimum size=3pt] at (7.75,0) {};
  \draw[implies]        (9,.25)   -- (9.5,-.25)node [below=7pt,text centered,midway]{\quad \tiny\begin{tabular}{l}$\Re(\alpha_{3}(c))$\\ $\Re(\alpha_{2}(c))$\end{tabular}};
  \node[circle,fill=black,inner sep=0pt,minimum size=3pt] at (9.25,0) {};
  \node at (10.75,0) [below=10pt] {$\cdots$};
  \draw[implies]        (12,.25)   -- (12.5,-.25) node [below=7pt,text centered,midway]{\tiny\begin{tabular}{l}$\Re(\alpha_{n}(c))$\\ $\Re(\alpha_{n-1}(c))$\end{tabular}};
   \node[circle,fill=black,inner sep=0pt,minimum size=3pt] at (12.25,0) {};
  \draw[] (6 ,-.1) -- (6 ,.1);
   \draw[-,dashed]        (-1.5,0)   -- (13.5,0) ; 
  \draw[->]        (-1.5,1)   -- (13.5,1) node [right=0pt]{\small $c$}; 
   \draw[->]        (6,-1)   -- (6,2) node [above=0pt]{\tiny $\Re(\alpha_j(c))$};
  \draw[-] (.75,.9) -- (.75,1.1) node [above=0pt]{\small $c_{n}^{(1)}$};
  \node at (2,1.1) [above=3pt] {$\cdots$};
  \draw[-] (3.75,.9) -- (3.75,1.1) node [above=0pt]{\small $c_{n}^{(\lfloor n/2\rfloor)}$};
  \draw[-] (7.75,.9) -- (7.75,1.1) node [above=0pt]{\small $c_{n}^{(\lceil n/2\rceil)}$};
   \draw[-] (9.25,.9) -- (9.25,1.1) node [above=0pt]{\small \quad $c_{n}^{(\lceil n/2\rceil+1)}$};
    \node at (11,1.1) [above=3pt] {$\cdots$};
   \draw[-] (12.25,.9) -- (12.25,1.1) node [above=0pt]{\small $c_{n}^{(n)}$};
\end{tikzpicture}
\caption{$n$ odd}
\end{subfigure}
\par\bigskip\medskip
\begin{subfigure}{0.90\textwidth}\lb{fig: n even-1}
\renewcommand\captionlabelfont{}
 \begin{tikzpicture}[baseline,remember picture,scale=0.8,
    implies/.style={double,double distance=.8pt,-implies},
    dot/.style={shape=circle,fill=black,minimum size=2pt,
                inner sep=0pt,outer sep=2pt},
]
  \draw[implies]        (-.5,-.25)   -- (0,.25) node [below=7pt,text centered,midway]{\tiny\begin{tabular}{l}$\Re(\alpha_{n-1}(c))$\\ $\Re(\alpha_{n-2}(c))$\end{tabular}};
  \node[circle,fill=black,inner sep=0pt,minimum size=3pt] at (-0.25,0) {};
  \node at (1.4,0) [below=10pt] {$\cdots$};
  \draw[implies]        (2.5,-.25)   -- (3,.25) node [below=7pt,text centered,midway]{\tiny\begin{tabular}{l}$\Re(\alpha_{3}(c))$\\ $\Re(\alpha_{2}(c))$\end{tabular}\ };
   \node[circle,fill=black,inner sep=0pt,minimum size=3pt] at (2.75,0) {};
  \draw[->]        (4,-.25)   -- (4.5,.25) node [below=7pt,text centered,midway]{\tiny\quad $\Re(\alpha_{1}(c))$};
  \node[circle,fill=black,inner sep=0pt,minimum size=3pt] at (4.25,0) {};
  \draw[implies]        (8,.25)   -- (8.5,-.25)node [below=7pt,text centered,midway]{\tiny\begin{tabular}{l}$\Re(\alpha_{2}(c))$\\ $\Re(\alpha_{1}(c))$\end{tabular}\ };
  \node[circle,fill=black,inner sep=0pt,minimum size=3pt] at (8.25,0) {};
  \draw[implies]        (9.5,.25)   -- (10,-.25) node [below=7pt,text centered,midway]{\tiny\quad \begin{tabular}{l}$\Re(\alpha_{3}(c))$\\ $\Re(\alpha_{2}(c))$\end{tabular}};
  \node[circle,fill=black,inner sep=0pt,minimum size=3pt] at (9.75,0) {};
  \node at (11.25,0) [below=10pt] {$\cdots$};
  \draw[implies]        (12.5,.25)   -- (13,-.25)node [below=7pt,text centered,midway]{\tiny\begin{tabular}{l}$\Re(\alpha_{n}(c))$\\ $\Re(\alpha_{n-1}(c))$\end{tabular}};
  \node[circle,fill=black,inner sep=0pt,minimum size=3pt] at (12.75,0) {};
  \draw[] (6,-.1) -- (6,.1);
     \draw[-,dashed]        (-1.5,0)   -- (13.5,0) ; 
  \draw[->]        (-1.5,1)   -- (13.5,1) node [right=0pt]{\small $c$}; 
  \draw[->]        (6,-1)   -- (6,2) node [above=0pt]
  {\tiny$\Re(\alpha_j(c))$};
  \draw[-] (-.25,.9) -- (-.25,1.1) node [above=0pt]{\small $c_{n}^{(1)}$};
  \node at (1.1,1.1) [above=3pt] {$\cdots$};
    \draw[-] (2.75,1.1) -- (2.75,1.1) node [above=0pt]{\small $c_{n}^{(\lfloor n/2\rfloor-1)}$};
  \draw[-] (4.25,.9) -- (4.25,1.1) node [above=0pt]{\small $c_{n}^{(\lfloor n/2\rfloor)}$};
  \draw[-] (8.25,.9) -- (8.25,1.1) node [above=0pt]{\small $c_{n}^{(\lfloor n/2\rfloor+1)}\ \ $};
  \draw[-] (9.75,.9) -- (9.75,1.1) node [above=0pt]{\small \quad $c_{n}^{(\lfloor n/2\rfloor+2)}$};
  \node at (11.5,1.1) [above=3pt] {$\cdots$};
   \draw[-] (12.75,.9) -- (12.75,1.1) node [above=0pt]{\small $c_{n}^{(n)}$};
\end{tikzpicture}
\caption{$n$ even}
 \lb{fig: n even-2}
 \end{subfigure}
\caption{The constants $c_{n}^{(k)}$}\lb{fig: constants}
\end{figure}
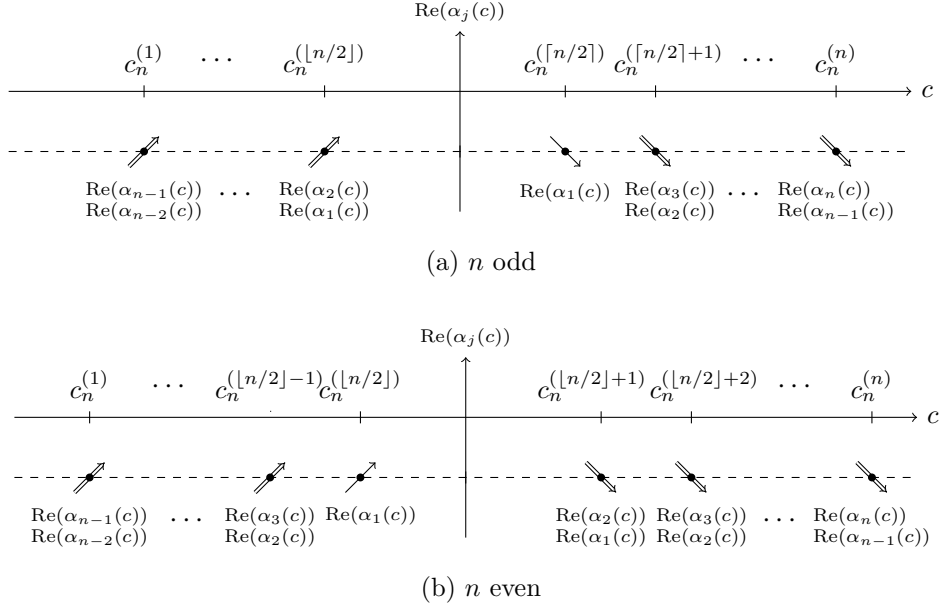 

\begin{corollary} \lb{c4.7}
For every $n\in \mathbb{N}$, there exists a positive constant $c_n\in \bbR$  such that  
\begin{equation}
\big\{c\in \mathbb{R} \,\big|\, \#\left(\tau_{2n}(c)|_{(0,d)}\right) =n \big\}=[c_n,\infty),
\end{equation}
and thus, 
\begin{align}
\begin{split} 
& \text{$T_{2n, min}(c)$ is self-adjoint $\big($equivalently, $\ddot T_{2n, min}$ is essentially 
self-adjoint\,$\big)$}     \\
& \quad \text{in $L^2((0,\infty);dx)$ if and only if \ $c \geq c_n$.}
\end{split}
\end{align} 
In addition, 
\begin{equation} 
c_1 = 3/4, \quad c_n = c_n^{(n)} \geq \f{(4n-1)!!}{2^{2n}}, \; n \in \bbN, \; n \geq 2  
\end{equation} 
$($see \eqref{4.27}, \eqref{4.28}, and \eqref{4.35}$)$. 
\end{corollary}

Put differently, Corollary \ref{c4.7} asserts there exist no ``islands'' $($i.e., intervals or its degeneration to points$)$ of non-essential self-adjointness for $\tau_{2n}(c)\big|_{C_0^{\infty}((0,\infty))}$ for $c \geq c_n$.

We explicitly record the following exact expressions: 
\begin{align}
c_{1} &= 3/4, \no \\
c_{2 }&= 45,  \no \\
c_{3 }&= 2240 \big(214+7 \sqrt{1009}\big)\big/27 \approx 36201.1645283357,   \no \\
c_{4 }&=2835 \Bigg( 13711+\f{190309441}{\sqrt[3]{2625188010911+1805760  
   \sqrt{-292868607}}}    \no \\
  &\quad\quad +  \sqrt[3]{2625188010911+1805760   \sqrt{-292868607}}\ \Bigg)   \\
  \begin{split} 
 &=38870685+5670 \sqrt{\f{292868607}{127}} \sin \left(\f{1}{3} \tan ^{-1}\left(\f{9
   \sqrt{292868607}}{466120}\right)\right)     \no \\
   &\quad\quad +\f{876128400}{\sqrt{127}} {\cos \left(\f{1}{3} \tan
   ^{-1}\left(\f{9 \sqrt{292868607}}{466120}\right)\right)}\\
 &\approx 117089256.9368802.     \no 
 \end{split} 
\end{align}

\begin{corollary} \lb{c4.8} 
For every  $n\in \mathbb{N}$ and every  $m\in \{n,n+1,\cdots,2n\}$, there exists some  $c\in \mathbb{R}$ such that $\#\left(\tau_{2n}(c)|_{(0,d)}\right)=m$.
\end{corollary}
\begin{proof}
By Theorem \ref{t4.6}, as $c$ increases from $c\ll0$ to $c\gg 0$,
$\#\left(\tau_{2n}(c)|_{(0,d)}\right)$ takes on the values
\begin{equation}
n+1,n+3, \ldots ,2n-2, 2n, 2n-1, 2n-3, \ldots , n+2,n, \quad \mbox{if $n$ is odd,} 
\end{equation} 
and 
\begin{equation}
n+1,n+3, \ldots , 2n-3, 2n-1, 2n, 2n-2, \ldots , n+2,n,
\quad \mbox{if $n$ is even}.
\end{equation}
In either case, $\#\left(\tau_{2n}(c)|_{(0,d)}\right)$
takes on all integer values from $n$ to $2n$.
\end{proof}

In particular, Corollary \ref{c4.8} proves that every possible integer in the interval $[n,2n]$ in \eqref{3.23} is attained for some $c \in \bbR$. 

\begin{example} \lb{e4.9}
If $n=3$, then  $q_0=10395/64$ and 
\begin{align}
h_2(c)&=\left|
\begin{array}{ccccc}
 18 & 435 & {4881}/{8} & 0 & 0\\[5pt]
 -1 & -{505}/{4} & -{12139}/{16} & c-10395/64 & 0 \\[5pt]
 0 & 18 & 435 & {4881}/{8} & 0 \\[5pt]
 0 & -1 & -505/4 & -12139/16 & c-10395/64 \\[5pt]
 0 & 0 & 18 & 435 & 4881/8 \\
\end{array}
\right| \no \\[1pc]  
& = -5832 c^2+207083520 c+146313216000, \quad c \in \bbR.
\end{align}
The roots of $h_2(\dott)$ are $2240\left(214\pm 7 \sqrt{1009}\right)/27$.
Therefore, by Theorem \ref{t4.6} one finds
\begin{equation}
\#\left(\tau_{6}(c)|_{(0,d)}\right)=
\begin{cases}
3,& \mbox{if}\quad  2240 \left(214+7 \sqrt{1009}\right)/27\leq c;\\[5pt]
5,& \mbox{if}\quad  10395/64\leq c<  2240\left(214+7 \sqrt{1009}\right)/27;\\[5pt]
6,& \mbox{if}\quad 2240\left(214-7 \sqrt{1009}\right)/27< c<10395/64;\\[5pt]
4,& \mbox{if}\quad c\leq 2240\left(214-7 \sqrt{1009}\right)/27. 
\end{cases}
\end{equation} 
 \end{example}

\appendix 
\section{Some Conjectures} \lb{sA}

In this section, when dealing with polynomials, we will view them as elements in a polynomial ring as in abstract algebra. We will review some standard notational conventions and basic results.
Let $X$ be an indeterminate (formal symbol).  We denote by $\bbZ[X]$ (resp. $\bbQ[X]$)  the ring of polynomials in the indeterminate $X$ with coefficients in $\bbZ$ (resp. $\bbQ$). A polynomial
$f(X)\in \bbQ[X]$ is called irreducible, if it has postive degree   and it cannot be written as 
a product $f(X)=g(X)h(X)$,  where $g(X),h(X)\in \bbQ[X]$ are polynomials of degree strictly less than the degree of $f(X)$.
  
\begin{conjecture}\lb{coA.1}
For $n\in \mathbb{N}$, $n\geq 2$, the polynomial
\begin{equation}
g_{n-1}(X) := \f{(-1)^{\lfloor n/2 \rfloor}}{{(2 n^2)}^{n}}\, h_{n-1}(X) 
\end{equation}
is a monic irreducible polynomial in  $\mathbb{Q}[X]$ of degree $n-1$ with Galois group $S_{n-1}$. In particular, for $n\geq 6$, the constants $c_{n}^{(1)},c_{n}^{(2)},\ldots c_{n}^{(\lceil n/2 \rceil-1},c_{n}^{(\lceil n/2 \rceil+1)}, \ldots , c_{n}^{(n)}$ are algebraic numbers that are not expressible in radicals over $\mathbb{Q}$.
\end{conjecture}

\bigskip

\noindent
{\it Proof for $n=5$}.
We  have
\begin{equation}
\begin{split}
g_4(X)=\ &X^4-5237598744576 X^3/5-3477424021724410819117056X^2/3125\\
&+2933863158888223380395161288704X/125\\
&+246639641224100448713004224731938816/55  .
\end{split}
\end{equation}   
Let $\widetilde{g}_4(X):=(3125)^4\, g_4(X/3125)$. Then $\widetilde{g}_4(X)$ is a monic polynomial of degree $4$ with integer coefficients. Reducing the coefficient of modulo~19, one obtains
\begin{equation}
\widetilde{g}_4(X)\equiv X^4+11 X^3+3 X^2+11 X+15\ \operatorname{mod} 19.
\end{equation}
It is easy to check that  $X^4+11 X^3+3 X^2+11X+15$ is irreducible modulo~19.
By Gauss' lemma, it follows that $g_4(X)$ is irreducible over $\bbQ$. \hfill $\Box$

\bigskip
\noindent
{\it Proof for $n=6$}. We  have
\begin{equation}
\begin{split}
g_5(X)=\ &X^5-15354318108567042605X^4/729\\[5pt]
&-333441081709503846926848000000X^3/3\\[5pt]
& +4983404391409567436628431599042560000000 X^2\\[5pt]
&+8770826733513986444066497798757941248000000000000X\\[5pt]
&-2088913117666248881257824386993081779822264320000000000000\!\!\!\!\!\!\!\!\!\!\!\!
\end{split}
\end{equation}
Let $\widetilde{g}_5(X):=(729)^5\, g_5(X/729)$. Then $\widetilde{g}_5(X)$ is a monic polynomial of degree $5$ with integer coefficients.
Note that $g_5(X)$ is irreducible over $\mathbb{Q}$ if and only if $\widetilde{g}_5(X)$ is irreducible over $\mathbb{Q}$.
Furthermore, the Galois group of $g_5(X)$ is  isomorphic to the Galois group of $\widetilde{g}_5(X)$.
To prove the irreducibility and to compute the Galois group, we reduce the coefficients of $\widetilde{g}_5(X)$ modulo the primes $23$ and $109$:
\begin{align}
 \widetilde{g}_5(X)&\equiv  X^5+5X^4+11 X^3+7 X^2+13X+16\ \operatorname{mod} 23,\\
  \widetilde{g}_5(X)&\equiv (X^2+38 X+24)(X+42)(X+41)(x+11)\ \operatorname{mod} 109.
\end{align}
It is easy to check that  $X^5+5 X^4+11 X^3+7 X^2+13X+16$ is irreducible modulo~23. Therefore, the polynomial   $\widetilde{g}_5(X)$
is irreducible over $\mathbb{Z}$ and also over $\mathbb{Q}$ by Gauss' lemma.
It also follows, by a theorem due to Dedekind  (see \cite[Thm.~4.37]{Ja85}), that the Galois group of the polynomial $\widetilde{g}_5(X)$ contains a 5-cycle.
Since  the reduction of $\widetilde{g}_5(X)$ modulo 109 is the product an irreducible quadratic polynomial and three linear polynomials,   Dedekind's theorem implies that the Galois group of  $\widetilde{g}_5(X)$ contains a transposition (2-cycle). 
A subgroup of $S_5$ that contains a transposition and a $5$-cycle is $S_5$. Since $S_5$ is not a solvable group,
Galois' theorem then implies  that $g_5(X)$ is not solvable and hence $c_{6}$ cannot be written in terms of radicals.  
\hfill $\Box$

\bigskip

\noindent
{\it Proof for $n=7$}. The least common denominator of the coefficients of $g_6(X)$ turns out to be $823543$. Let  $\widetilde{g}_6(X):=(823543)^6\, g_6(X/823543)$.
Then  $\widetilde{p}_6(X)$ is a monic polynomial   with integer coefficients and factorizations of $\widetilde{g}_6(X)$ modulo the primes 
$37$, $43$, and $89$ are
\begin{align}
 \widetilde{g}_6(X)&\equiv  (X^6+4 X^5+25X^4+20X^3+16 X^2+34 X+8)\ \operatorname{mod} 37,\\
 \widetilde{g}_6(X)&\equiv (X^2+15 X+5)(X+27) (X+20) (X+19) (X+9)\ \operatorname{mod} 43,\\
 \widetilde{g}_6(X)&\equiv (X^5+46 X^4+4 X^3+23 X^2+46 X+50)(X+30) \ \operatorname{mod} 89.
\end{align}
The factorization $\widetilde{g}_6(X)$ modulo $37$ reveals that $\widetilde{g}_6(X)$ is  irreducible over $\mathbb{Q}$.
\hfill $\Box$

\bigskip

The same idea can be used to prove the conjecture for larger $n$. The following table shows 
what primes are used to verify the conjecture for $4\leq n\leq 12$.
 
 \begin{table}[h]
\begin{tabular}{cccc}
\toprule 
$n$ & \multicolumn{3}{l}{Smallest prime needed to prove existence of}\\
  & $(n-1)$-cycle & $(n-2)$-cycle & $2$-cycle\\
\midrule
4 & 23 & 13 & 13\\
5 & 19 & 17 & 71\\
6 & 23 & 47 & 109\\
7 & 37 & 89 & 43\\
8 & 67 & 29 & 8089\\
9 & 179 & 47 & 7639\\
10 & 43 & 167 & 11519\\
11 & 59 & 41 & 2651743\\
12 & 53 & 67 & 19419221\\
\bottomrule
\end{tabular}
\caption{\lb{T.A.1}}
\end{table}

%

We conclude with a vexing open conjecture: 
\begin{conjecture} \lb{coA.2}
We have $\big($recalling $c_n = c_n^{(n)}$$\big)$
\begin{equation}
c_{n}  \underset{n \to \infty}{\sim} \big(2n^2\big/\pi\big)^{2n}.
\end{equation}
\end{conjecture}
\noindent 
{\it Sketch of the underlying idea.} By \eqref{4.14}, one infers that 
\begin{equation}
[\Re(\beta_n(c)) - (n - (1/2)) + i \Im(\beta_n(c))]^{2n} = e^{i\pi} e^{i(\pi/2) 2n} \big[c^{1/(2n)}\big]^{2n}, 
\end{equation}
and hence,  
\begin{equation}
\Re(\beta_n(c)) - (n - (1/2)) = - \sin(\pi/(2n)) c^{1/(2n)}. 
\end{equation}
Observing that $\Re(\alpha_n(c_n)) \approx \Re(\beta_n(c_n))$ for $c \gg 0$, one arrives at 
\begin{equation}
\Re(\alpha_n(c)) \approx (n-(1/2))- c^{1/(2n)} \sin(\pi/(2n)) \, \text{ for $c \gg 0$.}
\end{equation}
Finally, recalling that $\Re(\alpha_n(c_n))=-1/2$, and assuming that $c_n$ increases sufficiently rapidly with increasing $n$,
\begin{equation}
 (c_n)^{1/(2n)} \sin(\pi/(2n)) \approx n  , \text{ for $n \gg 0$.}
\end{equation}
Therefore, we {\it expect} (see Table \ref{T.A.2} below)  that
\begin{equation}
c_{n} \underset{n\to\infty}{\sim}  \big(2n^2\big/\pi\big)^{2n}.
\end{equation}

\begin{table} [h] 
\caption{\label{T.A.2} Asymptotic behavior of $c_n$ and $(2n^2/\pi)^{2n}$}
\begin{tabular}{rll}
\toprule
$n$ & c$_{n}$ & $(2n^2/\pi)^{2n}$\\
\midrule
 $1$ & $3/4$ & $0.40529$  \\
 $2$ & $45$ & $42.0495$  \\
 $3$ & $36201.2$ & $35378.2$  \\
 $4$ & $1.17089\times 10^8$ & $1.15878\times 10^8$  \\
 $5$ & $1.04858\times 10^{12}$ & $1.04280\times 10^{12}$  \\
 $6$ & $2.10674\times 10^{16}$ & $2.09987\times 10^{16}$  \\
 $7$ & $8.27892\times 10^{20}$ & $8.26165\times 10^{20}$  \\
 $8$ & $5.77530\times 10^{25}$ & $5.76715\times 10^{25}$  \\
 $9$ & $6.65283\times 10^{30}$ & $6.64619\times 10^{30}$  \\
 $10$ & $1.19652\times 10^{36}$ & $1.19565\times 10^{36}$  \\
 $11$ & $3.21278\times 10^{41}$ & $3.21100\times 10^{41}$  \\
 $12$ & $1.24167\times 10^{47}$ & $1.24115\times 10^{47}$  \\
$13$ & $6.70013\times 10^{52}$ & $6.69788\times 10^{52}$  \\
$14$ & $4.91961\times 10^{58}$ & $4.91828\times 10^{58}$  \\
$15$ & $4.80811\times 10^{64}$ & $4.80706\times 10^{64}$  \\
$16$ & $6.13651\times 10^{70}$ & $6.13540\times 10^{70}$  \\
$17$ & $1.00581\times 10^{77}$ & $1.00566\times 10^{77}$  \\
$18$ & $2.08622\times 10^{83}$ & $2.08595\times 10^{83}$  \\
$19$ & $5.40462\times 10^{89}$ & $5.40404\times 10^{89}$  \\
$20$ & $1.72840\times 10^{96}$ & $1.72824\times 10^{96}$  \\
$21$ & $6.75182\times 10^{102}$ & $6.75128\times 10^{102}$  \\
$22$ & $3.19118\times 10^{109}$ & $3.19096\times 10^{109}$  \\
$23$ & $1.80914\times 10^{116}$ & $1.80903\times 10^{116}$  \\
$24$ & $1.22053\times 10^{123}$ & $1.22046\times 10^{123}$  \\
$25$ & $9.72809\times 10^{129}$ & $9.72763\times 10^{129}$  \\
$26$ & $9.09940\times 10^{136}$ & $9.09902\times 10^{136}$  \\
$27$ & $9.92726\times 10^{143}$ & $9.92689\times 10^{143}$  \\
$28$ & $1.25603\times 10^{151}$ & $1.25599\times 10^{151}$  \\
$29$ & $1.83328\times 10^{158}$ & $1.83322\times 10^{158}$  \\
$30$ & $3.07164\times 10^{165}$ & $3.07155\times 10^{165}$  \\
$31$ & $5.88069\times 10^{172}$ & $5.88055\times 10^{172}$  \\
$32$ & $1.28096\times 10^{180}$ & $1.28093\times 10^{180}$  \\
$33$ & $3.16182\times 10^{187}$ & $3.16175\times 10^{187}$  \\
$34$ & $8.81027\times 10^{194}$ & $8.81010\times 10^{194}$  \\
$35$ & $2.76148\times 10^{202}$ & $2.76143\times 10^{202}$  \\
$36$ & $9.70367\times 10^{209}$ & $9.70351\times 10^{209}$  \\
$37$ & $3.81058\times 10^{217}$ & $3.81052\times 10^{217}$  \\
$38$ & $1.66725\times 10^{225}$ & $1.66723\times 10^{225}$  \\
$39$ & $8.10464\times 10^{232}$ & $8.10454\times 10^{232}$  \\
 \bottomrule
\end{tabular}
\end{table}

\newpage 

The following Table \ref{T.A.3} seems to suggest that for every $n\in \mathbb{N}$,
\begin{equation}
2n^2\big/\pi  < \big(c_{n}^{(n)}\big)^{1/(2n)} < n /\sin(\pi/(2n)) .
\end{equation}

\begin{table} [h]
\begin{tabular}{cccc}
\toprule
$n$ & $2n^2\big/\pi$ & $(c_{n})^{1/2n}$ & $n/\sin(\pi/(2n))$\\
\midrule
 1 & 0.6366198 & 0.8660254 & 1.0000000\\
 2 & 2.5464791 & 2.5900201 & 2.8284271\\
 3 & 5.7295780 & 5.7515790 & 6.0000000\\
 4 & 10.185916 & 10.199165 & 10.452504\\
 5 & 15.915494 & 15.924294 & 16.180340\\
 6 & 22.918312 & 22.924559 & 23.182220\\
 7 & 31.194369 & 31.199023 & 31.457714\\
 8 & 40.743665 & 40.747262 & 41.006647\\
 9 & 51.566202 & 51.569062 & 51.828934\\
 10 & 63.661977&63.664305 & 63.924532\\
 11 & 77.030992 &77.032923 & 77.293416\\
 12 & 91.673247 & 91.674874& 91.935571\\
 13 & 107.58874 & 107.59013 & 107.85099\\
 14 & 124.77748 & 124.77868 & 125.03966\\
 15 & 143.23945 & 143.24050 & 143.50158\\
 16 & 162.97466 & 162.97558 & 163.23676\\
 17 & 183.98311 & 183.98393 & 184.24517\\
 18 & 206.26481 & 206.26554 & 206.52684\\
 19 & 229.81974 & 229.82039 & 230.08175\\
 20 & 254.64791 & 254.64850 & 254.90990\\
 21 & 280.74932 & 280.74986 & 281.01129\\
 22 & 308.12397 & 308.12446 & 308.38593\\
 23 & 336.77186 & 336.77231 & 337.03380\\
 24 & 366.69299 & 366.69340 & 366.95492\\
 25 & 397.88736 & 397.88774 & 398.14928\\
 \bottomrule
\end{tabular}
\caption{\lb{T.A.3}}
\end{table}

\medskip

\noindent {\bf Acknowledgments.} We are indebted to Mark Ashbaugh, Andrei Martinez-Finkel- shtein, Alexander Sakhnovich, Rudi Weikard, and Maxim Zinchenko for very interesting discussions on this subject.  



\end{document}